\documentclass[high,12pt]{amsart}
\date{\today}

\input xymatrix
\xyoption{all}

\usepackage{amsmath,amsthm,amsfonts,amssymb,mathrsfs}

\usepackage{color}

\usepackage{amssymb, cite}
\usepackage[colorlinks,plainpages,citecolor=magenta, linkcolor=blue , backref]{hyperref}

\usepackage{hyperref}

\newtheorem{theorem}{Theorem}[section]
\newtheorem{lemma}[theorem]{Lemma}

\newtheorem{proposition}[theorem]{Proposition}
\newtheorem{corollary}[theorem]{Corollary}

\theoremstyle{definition}
\newtheorem{definition}[theorem]{Definition}

\theoremstyle{remark}
\newtheorem{remark}[theorem]{Remark}

\numberwithin{equation}{section}

\begin{document}

\title[On a locally compact monoid of cofinite partial isometries of $\mathbb{N}$ with adjoined zero]{On a locally compact monoid of cofinite partial isometries of $\mathbb{N}$ with adjoined zero}

\author[O.~Gutik and P.~Khylynskyi]{Oleg~Gutik and Pavlo~Khylynskyi}
\address{Faculty of Mathematics, Ivan Franko University of Lviv, Universytetska 1, Lviv, 79000, Ukraine}
\email{oleg.gutik@lnu.edu.ua, ogutik@gmail.com, pavlo.khylynskyi@lnu.edu.ua}



\begin{abstract}
Let $\mathscr{C}_\mathbb{N}$ be a monoid which is generated by the partial shift $\alpha\colon n\mapsto n+1$ of the set of positive integers $\mathbb{N}$ and its inverse partial shift $\beta\colon n+1\mapsto n$.
In this paper  we prove that if $S$ is a submonoid of the monoid $\mathbf{I}\mathbb{N}_{\infty}$ of all partial cofinite isometries of positive integers which contains $\mathscr{C}_\mathbb{N}$ as a submonoid then every Hausdorff locally compact shift-continuous topology on $S$ with adjoined zero is either compact or discrete. Also we show that the similar statement holds for a locally compact semitopological semigroup $S$ with  an adjoined compact ideal.
\end{abstract}

\maketitle


\section{Introduction and preliminaries}

In this paper we shall follow the terminology of \cite{Carruth-Hildebrant-Koch-1983, Carruth-Hildebrant-Koch-1986, Clifford-Preston-1961, Clifford-Preston-1967, Engelking-1989, Lawson-1998, Ruppert-1984}. The cardinality of a set $X$ is denoted by $|X|$.
By $\mathbb{N}$ and $\mathbb{Z}$ we denote the sets of positive integers and the set of all integers. Also we identify $\omega$ with the set $\mathbb{N}\cup\{0\}$.

For any subset $A$ of $\mathbb{N}$ and any $i\in\mathbb{N}$ we denote $i+A=\{i+k\colon k\in A\}$ when $A\neq\varnothing$ and $i+\varnothing=\varnothing$.

If $S$ is a semigroup, then we shall denote the subset of all
idempotents in $S$ by $E(S)$. If $S$ is an inverse semigroup, then
$E(S)$ is closed under multiplication and we shall refer to $E(S)$ as a
\emph{semilattice} (or the \emph{semilattice of} $S$).

A semigroup $S$ is called {\it inverse} if for any
element $x\in S$ there exists a unique $x^{-1}\in S$ such that
$xx^{-1}x=x$ and $x^{-1}xx^{-1}=x^{-1}$. The element $x^{-1}$ is
called the {\it inverse of} $x\in S$. If $S$ is an inverse
semigroup, then the function $\operatorname{inv}\colon S\to S$
which assigns to every element $x$ of $S$ its inverse element
$x^{-1}$ is called the {\it inversion}.
On an inverse semigroup $S$ the semigroup operation  determines the following partial order $\preccurlyeq$: $s\preccurlyeq t$ if and only if there exists $e\in E(S)$ such that $s=te$. This order is
called the {\em natural partial order} on $S$ and it induces the {\em natural partial order} on the semilattice $E(S)$ \cite{Wagner-1952}.
An inverse subsemigroup $T$ of an inverse semigroup $S$ is called \emph{full} if $E(T)=E(S)$.

A congruence $\mathfrak{C}$ on a semigroup $S$ is called a \emph{group congruence} if the quotient semigroup $S/\mathfrak{C}$ is a group. Any inverse semigroup $S$ admits the \emph{minimum group congruence} $\mathfrak{C}_{\mathbf{mg}}$:
\begin{equation*}
  a\mathfrak{C}_{\mathbf{mg}}b \quad \hbox{if and only if} \quad \hbox{there exists} \quad e\in E(S) \quad \hbox{such that} \quad ea=eb.
\end{equation*}
Also, we say that a semigroup homomorphism $\mathfrak{h}\colon S\to T$ is a \emph{group homomorphism} if the image $(S)\mathfrak{h}$ is a group, and $\mathfrak{h}\colon S\to T$ is \emph{trivial} if it is either an isomorphism or annihilating.

The bicyclic monoid ${\mathscr{C}}(p,q)$ is the semigroup with the identity $1$ generated by two elements $p$ and $q$ subjected only to the condition $pq=1$. The semigroup operation on ${\mathscr{C}}(p,q)$ is determined as
follows:
\begin{equation*}
    q^kp^l\cdot q^mp^n=q^{k+m-\min\{l,m\}}p^{l+n-\min\{l,m\}}.
\end{equation*}
It is well known that the bicyclic monoid ${\mathscr{C}}(p,q)$ is a bisimple (and hence simple) combinatorial $E$-unitary inverse semigroup and every non-trivial congruence on ${\mathscr{C}}(p,q)$ is a group congruence \cite{Clifford-Preston-1961}.

If $\alpha\colon X\rightharpoonup Y$ is a partial map, then we shall denote the domain and the range of $\alpha$ by $\operatorname{dom}\alpha$ and $\operatorname{ran}\alpha$, respectively. A partial map $\alpha\colon X\rightharpoonup Y$ is called \emph{cofinite} if both sets $X\setminus\operatorname{dom}\alpha$ and $Y\setminus\operatorname{ran}\alpha$ are finite.

The \emph{symmetric inverse} (\emph{monoid}) \emph{semigroup} $\mathscr{I}_\lambda$ over a cardinal $\lambda$ is the set of all partial one-to-one transformations of a non-zero  cardinal $\lambda$ endowed with the semigroup operation of composition of relations \cite{Clifford-Preston-1961}. The symmetric inverse semigroup was introduced by Wagner~\cite{Wagner-1952} and it plays a major role in the theory of semigroups. By $\mathscr{I}^{\mathrm{cf}}_\lambda$ is denoted a subsemigroup of injective partial selfmaps of $\lambda$ with
cofinite domains and ranges. Obviously, $\mathscr{I}^{\mathrm{cf}}_\lambda$ is an inverse
submonoid of the semigroup $\mathscr{I}_\lambda$. The
semigroup $\mathscr{I}^{\mathrm{cf}}_\lambda$  is called the \emph{monoid of
injective partial cofinite selfmaps} of $\lambda$ \cite{Gutik-Repovs-2015}.


A partial transformation $\alpha\colon (X,d)\rightharpoonup (X,d)$ of a metric space $(X,d)$ is called \emph{isometric} or a \emph{partial isometry}, if $d((x)\alpha,(y)\alpha)=d(x,y)$ for all $x,y\in \operatorname{dom}\alpha$. It is obvious that the set of partial cofinite isometries of a metric space $(X,d)$ with the operation of the composition of partial isometries is an inverse submonoid of the monoid of injective partial cofinite selfmaps of the cardinal $|X|$.


We endow the sets $\mathbb{N}$ and $\mathbb{Z}$ with the standard linear order.

The semigroup $\mathbf{ID}_{\infty}$ of all partial cofinite isometries of the set of integers $\mathbb{Z}$ with the usual metric $d(n,m)=|n-m|$, $n,m\in \mathbb{Z}$, was studied in  \cite{Bezushchak-2004, Bezushchak-2008, Gutik-Savchuk-2017}.

Let $\mathbf{I}\mathbb{N}_{\infty}$ be the set of all partial cofinite isometries of the set of positive integers $\mathbb{N}$ with the usual metric $d(n,m)=|n-m|$, $n,m\in \mathbb{N}$. Then $\mathbf{I}\mathbb{N}_{\infty}$ with the operation of composition of partial isometries is an inverse submonoid of $\mathscr{I}^{\mathrm{cf}}_\omega$. The semigroup $\mathbf{I}\mathbb{N}_{\infty}$ of all partial cofinite isometries of positive integers is studied in \cite{Gutik-Savchuk-2018}. There we described the Green relations on the semigroup $\mathbf{I}\mathbb{N}_{\infty}$, its band and proved that $\mathbf{I}\mathbb{N}_{\infty}$ is a simple $E$-unitary $F$-inverse semigroup \cite{Lawson-1998}. Also in \cite{Gutik-Savchuk-2018}, the least group congruence $\mathfrak{C}_{\mathbf{mg}}$ on $\mathbf{I}\mathbb{N}_{\infty}$ is described and there it is proved that the quotient-semigroup  $\mathbf{I}\mathbb{N}_{\infty}/\mathfrak{C}_{\mathbf{mg}}$ is isomorphic to the additive group of integers $\mathbb{Z}(+)$. In \cite{Gutik-Savchuk-2018} an example of a non-group congruence on the semigroup $\mathbf{I}\mathbb{N}_{\infty}$ is presented. Also it is  proved that a congruence on the semigroup $\mathbf{I}\mathbb{N}_{\infty}$ is a group congruence if and only if its restriction onto any isomorphic  copy of the bicyclic semigroup in $\mathbf{I}\mathbb{N}_{\infty}$ is a group congruence.
In \cite{Gutik-Savchuk-2020} it was shown that  every non-annihilating homomorphism $\mathfrak{h}\colon \mathbf{I}\mathbb{N}_{\infty}\to\mathbf{ID}_{\infty}$ has the following property:   the image $(\mathbf{I}\mathbb{N}_{\infty})\mathfrak{h}$ is isomorphic  either to the cyclic group $\mathbb{Z}_2$ of  order $2$ or the additive group of integers $\mathbb{Z}(+)$. Also it is proved that  $\mathbf{I}\mathbb{N}_{\infty}$ does not have a finite set of generators, and moreover it does not contain a minimal generating set.

Later by $\mathbb{I}$ we denote the unit element of $\mathbf{I}\mathbb{N}_{\infty}$.

\begin{remark}\label{remark-1.1}
We observe that the bicyclic semigroup is isomorphic to the
semigroup $\mathscr{C}_{\mathbb{N}}$ which is
generated by partial transformations $\alpha$ and $\beta$ of the set
of positive integers $\mathbb{N}$, defined as follows:
\begin{equation*}
\operatorname{dom}\alpha=\mathbb{N}, \qquad \operatorname{ran}\alpha=\mathbb{N}\setminus\{1\},  \qquad (n)\alpha=n+1
\end{equation*}
and
\begin{equation*}
\operatorname{dom}\beta=\mathbb{N}\setminus\{1\}, \qquad \operatorname{ran}\beta=\mathbb{N},  \qquad (n)\beta=n-1
\end{equation*}
(see Exercise~IV.1.11$(ii)$ in \cite{Petrich-1984}). It is obvious that $\mathbb{I}=\alpha\beta=\alpha^0=\beta^0$ and $\mathscr{C}_{\mathbb{N}}$ is a submonoid  of $\mathbf{I}\mathbb{N}_{\infty}$.
\end{remark}

The \emph{semigroup of cofinite monotone} (\emph{order preserving}) \emph{bijective partial transformations of} $\mathbb{N}$   was introduced in \cite{Gutik-Repovs-2011} and there it was denoted by $\mathscr{I}_{\infty}^{\!\nearrow}(\mathbb{N})$. Obviously, $\mathscr{I}_{\infty}^{\!\nearrow}(\mathbb{N})$ is an inverse subsemigroup of the semigroup $\mathscr{I}^{\mathrm{cf}}_\omega$.  In \cite{Gutik-Repovs-2011} Gutik and Repov\v{s} studied properties of the semigroup $\mathscr{I}_{\infty}^{\!\nearrow}(\mathbb{N})$. In particular, they showed that $\mathscr{I}_{\infty}^{\!\nearrow}(\mathbb{N})$ is an inverse bisimple semigroup and all of its non-trivial group homomorphisms are either isomorphisms or group homomorphisms. It is obvious that $\mathbf{I}\mathbb{N}_{\infty}$ is an inverse submonoid of $\mathscr{I}_{\infty}^{\!\nearrow}(\mathbb{N})$ \cite{Gutik-Savchuk-2018}.

A partial map $\alpha\colon \mathbb{N}\rightharpoonup \mathbb{N}$ is
called \emph{almost monotone} if there exists a finite subset $A$ of
$\mathbb{N}$ such that the restriction
$\alpha\mid_{\mathbb{N}\setminus A}\colon \mathbb{N}\setminus
A\rightharpoonup \mathbb{N}$ is a monotone partial map.
By $\mathscr{I}_{\infty}^{\,\Rsh\!\!\!\nearrow}(\mathbb{N})$ we
shall denote the \emph{semigroup of cofinite almost monotone
injective partial transformations of} $\mathbb{N}$.
Obviously, $\mathscr{I}_{\infty}^{\,\Rsh\!\!\!\nearrow}(\mathbb{N})$
is an inverse subsemigroup of the semigroup $\mathscr{I}^{\mathrm{cf}}_\omega$ and
the semigroup $\mathscr{I}_{\infty}^{\!\nearrow}(\mathbb{N})$ is an
inverse subsemigroup of
$\mathscr{I}_{\infty}^{\,\Rsh\!\!\!\nearrow}(\mathbb{N})$, too.
The semigroup $\mathscr{I}_{\infty}^{\,\Rsh\!\!\!\nearrow}(\mathbb{N})$ is studied in \cite{Chuchman-Gutik-2010}. In particular, it is shown that the semigroup
$\mathscr{I}_{\infty}^{\,\Rsh\!\!\!\nearrow}(\mathbb{N})$ is inverse,
bisimple and all of its non-trivial group homomorphisms are either
isomorphisms or group homomorphisms.  In the paper \cite{Gutik-Savchuk-2019} we showed that every automorphism of a full inverse subsemigroup of $\mathscr{I}_{\infty}^{\!\nearrow}(\mathbb{N})$ which contains the semigroup $\mathscr{C}_{\mathbb{N}}$ is the identity map. Also there we  constructed a submonoid $\mathbf{I}\mathbb{N}_{\infty}^{[\underline{1}]}$ of $\mathscr{I}_{\infty}^{\,\Rsh\!\!\!\nearrow}(\mathbb{N})$ with the following property: if $S$ be an inverse subsemigroup of $\mathscr{I}_{\infty}^{\,\Rsh\!\!\!\nearrow}(\mathbb{N})$ such that $S$ contains $\mathbf{I}\mathbb{N}_{\infty}^{[\underline{1}]}$ as a submonoid, then every non-identity congruence $\mathfrak{C}$ on $S$ is a group congruence. We show that if $S$ is an inverse submonoid of $\mathscr{I}_{\infty}^{\,\Rsh\!\!\!\nearrow}(\mathbb{N})$ such that $S$ contains $\mathscr{C}_{\mathbb{N}}$ as a submonoid then $S$ is simple and the quotient semigroup $S/\mathfrak{C}_{\mathbf{mg}}$  is isomorphic to the additive group of integers. Topologizations of inverse submonoids of $\mathscr{I}_{\infty}^{\,\Rsh\!\!\!\nearrow}(\mathbb{N})$  and embeddings of  such semigroups into compact-like topological semigroups are investigated in \cite{Chuchman-Gutik-2010, Gutik-Savchuk-2019}. Similar results for semigroups of cofinite almost monotone partial
bijections and cofinite almost monotone partial bijections of $\mathbb{Z}$ were obtained in \cite{Gutik-Repovs-2012}.

For an arbitrary positive integer $n_0$ we denote $[n_0)=\left\{n\in\mathbb{N}\colon n\geqslant n_0\right\}$. Since the set of all positive integers is well-ordered, the definition of the semigroup $\mathscr{I}_{\infty}^{\,\Rsh\!\!\!\nearrow}(\mathbb{N})$ implies that for every $\gamma\in\mathscr{I}_{\infty}^{\,\Rsh\!\!\!\nearrow}(\mathbb{N})$ there exists the smallest positive integer $n_{\gamma}^{\mathbf{d}}\in\operatorname{dom}\gamma$ such that the restriction $\gamma|_{\left[n_{\gamma}^{\mathbf{d}}\right)}$ of the partial map $\gamma\colon \mathbb{N}\rightharpoonup \mathbb{N}$ onto the set $\left[n_{\gamma}^{\mathbf{d}}\right)$ is an element of the semigroup $\mathscr{C}_{\mathbb{N}}$, i.e., $\gamma|_{\left[n_{\gamma}^{\mathbf{d}}\right)}$ is a shift of $\left[n_{\gamma}^{\mathbf{d}}\right)$. For every $\gamma\in\mathscr{I}_{\infty}^{\,\Rsh\!\!\!\nearrow}(\mathbb{N})$ we put $\overrightarrow{\gamma}=\gamma|_{\left[n_{\gamma}^{\mathbf{d}}\right)}$, i.e.
\begin{equation*}
\operatorname{dom}\overrightarrow{\gamma}=\big[n_{\gamma}^{\mathbf{d}}\big), \quad (x)\overrightarrow{\gamma}=(x)\gamma \quad \hbox{for all} \; x\in \operatorname{dom}\overrightarrow{\gamma} \quad \hbox{and} \quad \operatorname{ran}\overrightarrow{\gamma}=\left(\operatorname{dom}\overrightarrow{\gamma}\right)\gamma.
\end{equation*}
Also, we put
\begin{equation*}
\underline{n}_{\gamma}^{\mathbf{d}}=\min\operatorname{dom}\gamma \qquad \hbox{for} \quad \gamma\in\mathscr{I}_{\infty}^{\,\Rsh\!\!\!\nearrow}(\mathbb{N}).
\end{equation*}
It is obvious that $\underline{n}_{\gamma}^{\mathbf{d}}= n_{\gamma}^{\mathbf{d}}$ when $\gamma\in\mathscr{C}_{\mathbb{N}}$, and $\underline{n}_{\gamma}^{\mathbf{d}}< n_{\gamma}^{\mathbf{d}}$ when $\gamma\in\mathscr{I}_{\infty}^{\,\Rsh\!\!\!\nearrow}(\mathbb{N})\setminus\mathscr{C}_{\mathbb{N}}$. Also for any $\gamma\in\mathbf{I}\mathbb{N}_{\infty}$ we denote
\begin{equation*}
  \underline{n}_{\gamma}^{\mathbf{r}}=(\underline{n}_{\gamma}^{\mathbf{d}})\gamma=\min\operatorname{ran}\gamma \qquad \hbox{and} \qquad n_{\gamma}^{\mathbf{r}}=(n_{\gamma}^{\mathbf{d}})\gamma.
\end{equation*}

The results of Section~3 of \cite{Gutik-Savchuk-2020} imply that $n_{\gamma}^{\mathbf{r}}-\underline{n}_{\gamma}^{\mathbf{r}}=n_{\gamma}^{\mathbf{d}}-\underline{n}_{\gamma}^{\mathbf{d}}$ for any $\gamma\in\mathbf{I}\mathbb{N}_{\infty}$, and moreover for any non-negative integer $k$ we have that
\begin{equation*}
  \mathbf{I}\mathbb{N}_{\infty}^{[k]}=\big\{\gamma\in \mathbf{I}\mathbb{N}_{\infty}\colon n_{\gamma}^{\mathbf{d}}-\underline{n}_{\gamma}^{\mathbf{d}}\leqslant k\big\}
\end{equation*}
is a simple  inverse subsemigroup of $\mathbf{I}\mathbb{N}_{\infty}$ such that $\mathbf{I}\mathbb{N}_{\infty}$ admits the following infinite semigroup series
\begin{equation*}
\mathscr{C}_{\mathbb{N}}=\mathbf{I}\mathbb{N}_{\infty}^{[0]}= \mathbf{I}\mathbb{N}_{\infty}^{[1]}\subsetneqq \mathbf{I}\mathbb{N}_{\infty}^{[2]}\subsetneqq \mathbf{I}\mathbb{N}_{\infty}^{[3]}\subsetneqq \cdots \subsetneqq \mathbf{I}\mathbb{N}_{\infty}^{[k]}\subsetneqq \cdots \subset \mathbf{I}\mathbb{N}_{\infty}.
\end{equation*}

For any positive integer $k$ the semigroup $\mathbf{I}\mathbb{N}_{\infty}^{[k]}$ is called the \emph{monoid of cofinite isometries of positive integers with the noise} $k$.

A topological space $X$ is called:
\begin{itemize}
  \item \emph{Baire} if for each sequence $A_1, A_2,\ldots, A_i,\ldots$ of nowhere dense subsets of $X$ the union $\displaystyle\bigcup_{i=1}^\infty A_i$ is a co-dense subset of $X$~\cite{Engelking-1989, Haworth-McCoy-1977};
  \item \emph{locally compact} if every point $x$ of $X$ has an open neighbourhood with the compact closure.
\end{itemize}

A (\emph{semi})\emph{topological} \emph{semigroup} is a topological space endowed with a (separately) continuous semigroup operation. An inverse topological semigroup with continuous inversion is called a \emph{topological inverse semigroup}.

A topology $\tau$ on a semigroup $S$ is called:
\begin{itemize}
  \item a \emph{semigroup} topology if $(S,\tau)$ is a topological semigroup;
   \item an \emph{inverse semigroup} topology if $(S,\tau)$ is a topological inverse semigroup;
  \item a \emph{shift-continuous} topology if $(S,\tau)$ is a semitopological semigroup.
\end{itemize}

The bicyclic monoid admits only the discrete semigroup Hausdorff topology \cite{Eberhart-Selden-1969}. Bertman and  West in \cite{Bertman-West-1976} extended this result for the case of Hausdorff semitopological semigroups. Stable and $\Gamma$-compact topological semigroups do not contain the bicyclic monoid~\cite{Anderson-Hunter-Koch-1965, Hildebrant-Koch-1986, Koch-Wallace-1957}. The problem of embedding the bicyclic monoid into compact-like topological semigroups was studied in \cite{Banakh-Dimitrova-Gutik-2009, Banakh-Dimitrova-Gutik-2010, Bardyla-Ravsky-2019, Gutik-Repovs-2007}.

In the paper \cite{Gutik-P.Khylynskyi-2021} we studied algebraic properties of the monoid $\mathbf{I}\mathbb{N}_{\infty}^{[j]}$ and extend the Eberhard-Selden and Bertman-West results  \cite{Eberhart-Selden-1969, Bertman-West-1976} to the semigroups $\mathbf{I}\mathbb{N}_{\infty}^{[j]}$, $j\geqslant 0$. In particular it is shown that for any positive integer $j$ every Hausdorff shift-continuous  topology $\tau$ on $\mathbf{I}\mathbb{N}_{\infty}^{[j]}$ is discrete and and if $\mathbf{I}\mathbb{N}_{\infty}^{[j]}$ is a proper dense subsemigroup of a Hausdorff semitopological semigroup $S$, then $S\setminus \mathbf{I}\mathbb{N}_{\infty}^{[j]}$  is a closed ideal of $S$, and moreover if $S$ is a topological inverse semigroup then $S\setminus \mathbf{I}\mathbb{N}_{\infty}^{[j]}$ is a topological group. Also it is described the algebraic and topological structure of the closure of the monoid $\mathbf{I}\mathbb{N}_{\infty}^{[j]}$ in a locally compact topological inverse semigroup.

The well known A.~Weil Theorem states that \emph{every locally compact monothetic topological group $G$} (i.e., $G$ contains a cyclic dense subgroup) \emph{is either compact or discrete} (see \cite{Weil-1938}). A semitopological semigroup $S$ is called \emph{monothetic} if it contains a cyclic dense subsemigroup. Locally compact and compact monothetic topological semigroups were studied by Hewitt \cite{Hewitt-1956},  Hofmann \cite{Hofmann-1960}, Koch  \cite{Koch-1957}, Numakura \cite{Numakura-1952} and others (see more information on this topics in the books \cite{Carruth-Hildebrant-Koch-1986} and \cite{Hofmann-Mostert-1966}). Koch in \cite{Koch-1957} posed the following problem: ``\emph{If $S$ is a locally compact monothetic semigroup and $S$ has an identity, must $S$ be compact?}'' From the other side, Zelenyuk in \cite{Zelenyuk-1988} constructed a countable monothetic locally compact topological semigroup without unit which is neither compact nor discrete and in \cite{Zelenyuk-2019} he constructed a monothetic locally compact topological monoid with the same property. The topological properties of monothetic locally compact (semi)topological semigroups studied in \cite{Banakh-Bardyla-Guran-Gutik-Ravsky-2020, Guran-Kisil-2012, Zelenyuk-2020, Zelenyuk-Zelenyuk-2020}.

In the paper \cite{Gutik-2015} it is proved that every Hausdorff locally compact shift-continuous topology on the bicyclic monoid with adjoined zero is either compact or discrete. This result was extended by Bardyla onto the a polycyclic monoid \cite{Bardyla-2016} and graph inverse semigroups \cite{Bardyla-2018}, and by Mokrytskyi onto the monoid of order isomorphisms between principal filters of $\mathbb{N}^n$ with adjoined zero \cite{Mokrytskyi-2019}. Also, in \cite{Gutik-Maksymyk-2019} it is proved that the extended bicyclic semigroup $\mathscr{C}_\mathscr{\mathbb{Z}}^0$ with adjoined zero admits continuum many different  shift-continuous topologies, however every Hausdorff locally compact semigroup topology on $\mathscr{C}_\mathscr{\mathbb{Z}}^0$ is discrete. In \cite{Bardyla=2021??} Bardyla proved that a Hausdorff locally compact semitopological semigroup McAlister Semigroup $\mathcal{M}_1$ is either compact or discrete. However, this dichotomy does not hold for the McAlister Semigroup $\mathcal{M}_2$ and moreover, $\mathcal{M}_2$ admits continuum many different Hausdorff locally compact inverse semigroup topologies \cite{Bardyla=2021??}.

In this paper we extend the results of paper \cite{Gutik-2015} onto the monoid $\mathbf{I}\mathbb{N}_{\infty}$ of all partial cofinite isometries of positive integers with adjoined zero. In particular we prove that if $S$ is a submonoid of $\mathbf{I}\mathbb{N}_{\infty}$ which contains $\mathscr{C}_\mathbb{N}$ as a submonoid then every Hausdorff locally compact shift-continuous topology on $S$ with adjoined zero is either compact or discrete. Also we show that the similar statement holds for a locally compact semitopological semigroup $S$ with  an adjoined compact ideal.

\smallskip

\section{On submonoids of the monoid $\mathbf{I}\mathbb{N}_{\infty}$}\label{sec-2}

\begin{lemma}\label{lemma-2.1}
Let $S$ be a subsemigroup of $\mathbf{I}\mathbb{N}_{\infty}$ which contains $\mathscr{C}_{\mathbb{N}}$. Then $S$ is a simple inverse monoid.
\end{lemma}

\begin{proof}
The statement of the lemma is trivial when either $S=\mathbf{I}\mathbb{N}_{\infty}$ or $S=\mathscr{C}_{\mathbb{N}}$. Hence, we suppose that $S$ is a proper submonoid of $\mathbf{I}\mathbb{N}_{\infty}$ which contains $\mathscr{C}_{\mathbb{N}}$ as a proper submonoid. To prove that $S$ is an inverse semigroup, it is sufficient to show that every non-idempotent element of $S$ has inverse in $S$.

Let $\gamma$ be an arbitrary non-idempotent element of $S\setminus\mathscr{C}_{\mathbb{N}}$.
By Lemma~1 of \cite{Gutik-Savchuk-2018}, $\gamma$ is a shift of a subset $\operatorname{dom}\gamma$ of $\mathbb{N}$.
Put $\gamma_0=\beta^{\underline{n}_{\gamma}^{\mathbf{r}}}\alpha^{\underline{n}_{\gamma}^{\mathbf{d}}}$. Then $\operatorname{dom}\gamma\subseteq \operatorname{ran}\gamma_0$, $\operatorname{ran}\gamma\subseteq \operatorname{dom}\gamma_0$, and hence
$$
(i)\gamma_0^{-1}=(i)\beta^{\underline{n}_{\gamma}^{\mathbf{d}}}\alpha^{\underline{n}_{\gamma}^{\mathbf{r}}}=(i)\gamma
$$
for all $i\in \operatorname{dom}\gamma$. This implies that $\gamma\gamma_0$ is the identity map of the set $\operatorname{dom}\gamma$ and $\gamma_0\gamma$ is the identity map of the set $\operatorname{ran}\gamma$. Hence we get that the elements  $\gamma\gamma_0$ and $\gamma_0\gamma$ are idempotents in the semigroup $S$, and $\gamma\gamma_0\gamma=\gamma$. We claim that the element $\gamma_0\gamma\gamma_0$ is inverse of $\gamma$. Indeed, we have that
\begin{equation*}
  \gamma(\gamma_0\gamma\gamma_0)\gamma=(\gamma\gamma_0\gamma)\gamma_0\gamma=\gamma\gamma_0\gamma=\gamma
\end{equation*}
and
\begin{align*}
  (\gamma_0\gamma\gamma_0)\gamma(\gamma_0\gamma\gamma_0)&=\gamma_0(\gamma\gamma_0\gamma)(\gamma_0\gamma\gamma_0)=\\
     &=\gamma_0\gamma(\gamma_0\gamma\gamma_0)=\\
     &=\gamma_0(\gamma\gamma_0\gamma)\gamma_0=\\
     &=\gamma_0\gamma\gamma_0.
\end{align*}
By Theorem~5 of \cite{Gutik-Savchuk-2019} the monoid $S$ is simple.
\end{proof}

Fix an arbitrary non-empty finite subset $A$ of $\mathbb{N}$ such that $1\in A$. Let $n^0$ be a positive integer such  that $n^0\geqslant \max A+2$. We denote $A[n^0)=A\cup[n^0)$ and for any $i\in\omega$ we put
\begin{equation*}
  i+A[n^0)=\{i+k\colon k\in A[n^0)\}.
\end{equation*}

By Lemma~1 of \cite{Gutik-Savchuk-2018} every element of the semigroup $\mathbf{I}\mathbb{N}_{\infty}$ is a partial shift of a co-finite subset of $\mathbb{N}$. Hence in the above presented terminology we get the following descriptions of elements of the semigroup $\mathbf{I}\mathbb{N}_{\infty}$.

\begin{corollary}\label{corollary-2.3}
For every element $\gamma$ of the semigroup $\mathbf{I}\mathbb{N}_{\infty}$ one of the following conditions holds:
\begin{itemize}
  \item[$(i)$] $\gamma$ is a shift of the set $[n^0)$ for some positive integer $n^0$, i.e., $\mathrm{dom}\gamma=[n^0)$;
  \item[$(ii)$] $\gamma$ is a shift of the set $i+A[n^0)$ for some $i\in\omega$, $n^0\in\mathbb{N}$ and a non-empty finite subset $A\subset \mathbb{N}$ such that $\min A=1$ and $n^0\geqslant\max A+2$, i.e., $\mathrm{dom}\gamma=i+A[n^0)$.
\end{itemize}
\end{corollary}

For any non-empty finite subset $A\subset \mathbb{N}$ such that $\min A=1$, for any nonnegative integer $i$ and any positive integer $n^0\geqslant\max A+2$ by $\varepsilon^{n^0}_A[i)$ we denote the identity map of the set $i+A[n^0)$.

Corollary~\ref{corollary-2.3} implies

\begin{proposition}\label{proposition-2.4}
Let $\gamma$ be an element of the semigroup $\mathbf{I}\mathbb{N}_{\infty}$ such that $\mathrm{dom}\gamma=i+A[n^0)$ for some non-empty finite subset $A\subset \mathbb{N}$, a nonnegative integer $i$ and a positive integer $n^0\geqslant\max A+2$. Then there exists a nonnegative integer $j$ such that
\begin{equation}\label{eq-2.1}
  \gamma=\varepsilon^{n^0}_A[i)\cdot \beta^i\alpha^j=\beta^i\alpha^j\cdot\varepsilon^{n^0}_A[j).
\end{equation}
\end{proposition}

\begin{remark}\label{remark-2.5}
We observe that  representation \eqref{eq-2.1} of a such element $\gamma$ is not unique. Moreover for any non-negative integer $k$ such that $i-k\geqslant 0$ and  $j-k\geqslant 0$ we have that
\begin{equation*}
  \gamma=\varepsilon^{n^0}_A[i)\cdot \beta^{i-k}\alpha^{j-k}=\beta^{i-k}\alpha^{j-k}\cdot\varepsilon^{n^0}_A[j).
\end{equation*}
Also, since in formula \eqref{eq-2.1} we have that $i=\underline{n}_{\gamma}^{\mathbf{d}}$ and $j=\underline{n}_{\gamma}^{\mathbf{r}}$, the equalities \eqref{eq-2.1} are called the \emph{canonical representations} of $\gamma\in \mathbf{I}\mathbb{N}_{\infty}$.
\end{remark}

On any submonoid $S$ of the semigroup $\mathbf{I}\mathbb{N}_{\infty}$ which contains $\mathscr{C}_\mathbb{N}$ we define a binary relation $\ll$ in the following way:
\begin{equation*}
  \gamma\ll\delta \quad \hbox{if and only if there exists a non-negative integer~} i  \hbox{~such that~}\gamma=\beta^i\delta\alpha^i,
\end{equation*}
for $\gamma,\delta\in S$.

\begin{proposition}\label{proposition-2.6}
The binary relation $\ll$ is a partial order on any submonoid $S$ of the semigroup $\mathbf{I}\mathbb{N}_{\infty}$ which contains $\mathscr{C}_\mathbb{N}$.
\end{proposition}

\begin{proof}
Obviously that  $\gamma=\beta^0\gamma\alpha^0$ for any $\gamma\in S$, and hence $\ll$ is reflexive.

Suppose that $\gamma\ll\delta$ and $\delta\ll\gamma$ for some $\gamma,\delta\in S$. Then $\gamma=\beta^i\delta\alpha^i$ and $\delta=\beta^j\gamma\alpha^j$ for some $i,j\in\omega$. This implies that $\gamma=\beta^i\beta^j\gamma\alpha^j\alpha^i=\beta^{i+j}\gamma\alpha^{j+i}$. Then the definitions of elements $\alpha$ and $\beta$ imply that $\underline{n}_{\gamma}^{\mathbf{d}}\leqslant \underline{n}_{\beta^{i+j}\gamma\alpha^{j+i}}^{\mathbf{d}}$, and hence we get that $i+j=0$. Then $\gamma=\delta$, and hence the binary relation $\ll$ is antisymmetric.

Suppose that $\gamma\ll\delta$ and $\delta\ll\eta$ for some $\gamma,\delta,\eta\in S$, i.e., $\gamma=\beta^i\delta\alpha^i$ and $\delta=\beta^j\eta\alpha^j$ for some $i,j\in\omega$. This implies that $\gamma=\beta^i\beta^j\eta\alpha^j\alpha^i=\beta^{i+j}\gamma\alpha^{j+i}$. Thus $\gamma\ll\eta$ and hence the binary relation $\ll$ is transitive.
\end{proof}

Some interesting property of the relation $\ll$ on the semigroup $\mathbf{I}\mathbb{N}_{\infty}$ presents the following proposition.

\begin{proposition}\label{proposition-2.7}
The restriction of the binary relation $\ll$ on the subsemigroup $\mathscr{C}_\mathbb{N}$ of $\mathbf{I}\mathbb{N}_{\infty}$ is the natural partial order on $\mathscr{C}_\mathbb{N}$.
\end{proposition}

\begin{proof}
Suppose that $\beta^{i_1}\alpha^{j_1}\ll\beta^{i_2}\alpha^{j_2}$ for some $\beta^{i_1}\alpha^{j_1},\beta^{i_2}\alpha^{j_2}\in \mathscr{C}_\mathbb{N}$. The definition of the relation $\ll$ implies that there exists $i\in\omega$ such  that $i_1=i_2+i$ and $j_1=j_2+i$. Then we have that
\begin{align*}
  \beta^{i_2}\alpha^{j_2}\cdot \beta^{j_2+i}\alpha^{j_2+i}&= \beta^{i_2}\alpha^{j_2}(\beta^{j_2} \beta^{i})\alpha^{j_2+i}=\\
    &=\beta^{i_2}(\alpha^{j_2}\beta^{j_2})\beta^{i}\alpha^{j_2+i}=\\
    &=\beta^{i_2}\beta^{i}\alpha^{j_2+i}= \\
    &=\beta^{i_2+i}\alpha^{j_2+i}=\\
    &=\beta^{i_1}\alpha^{j_1},
\end{align*}
and hence $\beta^{i_1}\alpha^{j_1}\preccurlyeq\beta^{i_2}\alpha^{j_2}$ in $\mathscr{C}_\mathbb{N}$.

Suppose that $\beta^{i_1}\alpha^{j_1}\preccurlyeq\beta^{i_2}\alpha^{j_2}$ for some $\beta^{i_1}\alpha^{j_1},\beta^{i_2}\alpha^{j_2}\in \mathscr{C}_\mathbb{N}$. By Lemma~1.4.6 of \cite{Lawson-1998} we have that
\begin{align*}
\beta^{i_1}\alpha^{j_1}
  &=\beta^{i_2}\alpha^{j_2}\cdot (\beta^{i_1}\alpha^{j_1})^{-1}\cdot \beta^{i_1}\alpha^{j_1}= \\
  &=\beta^{i_2}\alpha^{j_2}\cdot \beta^{j_1}\alpha^{i_1} \beta^{i_1}\alpha^{j_1}= \\
  &=\beta^{i_2}\alpha^{j_2}\cdot \beta^{j_1}\alpha^{j_1}= \\
  &=\left\{
      \begin{array}{ll}
        \beta^{i_2-j_2+j_1}\alpha^{j_1}, & \hbox{if~} j_2\leqslant j_1;\\
        \beta^{i_2}\alpha^{j_2}, & \hbox{if~} j_2>j_1.
      \end{array}
    \right.
\end{align*}
This implies $j_2\leqslant j_1$ and $i_2-j_2+j_1=i_1$, and hence $k=i_1-i_2=j_1-j_2\geqslant 0$. Then
\begin{equation*}
  \beta^k\cdot\beta^{i_1}\alpha^{j_1}\cdot\alpha^k=\beta^{i_1-i_2}\cdot\beta^{i_2}\alpha^{j_2}\cdot\alpha^{j_1-j_2}=
  \beta^{i_1-i_2+i_2}\alpha^{j_2+j_1-j_2}=\beta^{i_1}\alpha^{j_2},
\end{equation*}
which implies that $\beta^{i_1}\alpha^{j_1}\ll\beta^{i_2}\alpha^{j_2}$ in $\mathbf{I}\mathbb{N}_{\infty}$.
\end{proof}

For any $\gamma\in\mathbf{I}\mathbb{N}_{\infty}$ and $M\subseteq \mathbf{I}\mathbb{N}_{\infty}$ we denote
\begin{equation*}
{\downarrow_{\ll}}\gamma=\left\{\delta\in\mathbf{I}\mathbb{N}_{\infty}\colon \delta\ll\gamma\right\} \qquad \hbox{and} \qquad {\downarrow_{\ll}}M=\displaystyle\bigcup_{\delta\in M}{\downarrow_{\ll}}\delta.
\end{equation*}

A linearly ordered set $(X,\eqslantless)$ is called an \emph{$\omega$-chain} if it is order isomorphic to $(\omega,\geqslant)$, where $\geqslant$ is the dual order to the usual linear order $\leqslant$ on $\omega$ \cite{Reilly-1966}.

\begin{lemma}\label{lemma-2.8}
Let $S$ be a subsemigroup of $\mathbf{I}\mathbb{N}_{\infty}$ which contains $\mathscr{C}_\mathbb{N}$. Then for any  $\gamma\in S$, ${\downarrow_{\ll}}\gamma\subset $S and the poset $({\downarrow_{\ll}}\gamma,\ll)$ is an $\omega$-chain.
\end{lemma}

\begin{proof}
We define a mapping $\mathfrak{f}\colon \omega\to {\downarrow_{\ll}}\gamma$ by the formula $\mathfrak{f}(i)=\beta^i\gamma\alpha^i$. Simple verifications show $\mathfrak{f}$ is an order isomorphism of $(\omega,\geqslant)$ onto $({\downarrow_{\ll}}\gamma,\ll)$.

This implies that ${\downarrow_{\ll}}\gamma=\left\{\delta\in\mathbf{I}\mathbb{N}_{\infty}\colon \delta\ll\gamma\right\}$ is a subset of $S$.
\end{proof}

The proof of Lemma~\ref{lemma-2.9} is obvious and it follows from the definition of the set ${\downarrow_{\ll}}\gamma$ for any $\gamma\in\mathbf{I}\mathbb{N}_{\infty}$.

\begin{lemma}\label{lemma-2.9}
Let $S$ be a subsemigroup of $\mathbf{I}\mathbb{N}_{\infty}$ which contains $\mathscr{C}_\mathbb{N}$. Then for every $\gamma\in S$ the following statements hold:
\begin{enumerate}
  \item\label{lemma-2.9-1} $\beta^i\cdot{\downarrow_{\ll}}\gamma\cdot\alpha^i\subseteq {\downarrow_{\ll}}\gamma$ for any $i\in\omega$;
  \item\label{lemma-2.9-2} ${\downarrow_{\ll}}\gamma\subseteq {\downarrow_{\ll}}(\alpha^i\gamma\beta^i)$ if and only if $\underline{n}_{\gamma}^{\mathbf{d}}-i\geqslant 0$ and $\underline{n}_{\gamma}^{\mathbf{r}}-i\geqslant 0$.
\end{enumerate}
\end{lemma}

It is easy to check that the elements $\gamma\in\mathbf{I}\mathbb{N}_{\infty}$ such that $\underline{n}_{\gamma}^{\mathbf{d}}=0$ or $\underline{n}_{\gamma}^{\mathbf{r}}=0$ are maximal in the poset $(\mathbf{I}\mathbb{N}_{\infty},\ll)$.

If  $A=\varnothing$ then we put $n^0=0$ and $\left\langle A[n^0)\right\rangle=\mathscr{C}_\mathbb{N}$.
Also, for any non-empty finite subset $A\subset \mathbb{N}$ such that $\min A=1$ and for any positive integer $n^0\geqslant 2+\max A$ we denote
\begin{equation*}
  \left\langle A[n^0)\right\rangle=\left\{\varepsilon^{n^0}_A[i)\cdot \beta^i\alpha^j\colon i,j\in\omega\right\}.
\end{equation*}

The proof of Lemma~\ref{lemma-2.1} and Proposition~\ref{proposition-2.4} imply the following proposition.

\begin{proposition}\label{proposition-2.11}
Let $S$ be a subsemigroup of $\mathbf{I}\mathbb{N}_{\infty}$ which contains $\mathscr{C}_\mathbb{N}$. Let $A$ be a non-empty finite subset of $\mathbb{N}$ such that $\min A=1$ and let $n^0$ be a positive integer such that $n^0\geqslant 2+\max A$. If $\left\langle A[n^0)\right\rangle\cap S\neq\varnothing$ then $\left\langle A[n^0)\right\rangle\subseteq S$.
\end{proposition}

The proof of the following lemma is a routine verification which follows from the definition of an idempotent $\varepsilon^{n^0}_A[i)\in \mathbf{I}\mathbb{N}_{\infty}$.

\begin{lemma}\label{lemma-2.12}
For an arbitrary idempotent $\varepsilon^{n^0}_A[i)\in\mathbf{I}\mathbb{N}_{\infty}{\setminus}\mathscr{C}_\mathbb{N}$ the following conditions hold:
\begin{enumerate}
  \item\label{lemma-2.12-1} $\alpha^p\cdot\varepsilon^{n^0}_A[i)=\varepsilon^{n^0}_A[i-p)\cdot\alpha^p$ for any non-negative integer $p\leqslant i$;
  \item\label{lemma-2.12-2} $\beta^q\cdot\varepsilon^{n^0}_A[i)=\varepsilon^{n^0}_A[i+q)\cdot\beta^q$ for any non-negative integer $q$;
  \item\label{lemma-2.12-3} $\varepsilon^{n^0}_A[i)\cdot\alpha^p=\alpha^p\cdot\varepsilon^{n^0}_A[i+p)$ for any non-negative integer $p$;
  \item\label{lemma-2.12-4} $\varepsilon^{n^0}_A[i)\cdot\beta^q=\beta^q\cdot\varepsilon^{n^0}_A[i-q)$ for any non-negative integer $q\leqslant i$
\end{enumerate}
\end{lemma}

\begin{proposition}\label{proposition-2.13}
Let $A$ be a non-empty finite subset of $\mathbb{N}$ such that $\min A=1$ and let $n^0$ be a positive integer such that $n^0\geqslant 2+\max A$. Then ${\downarrow_{\ll}}\left\langle A[n^0)\right\rangle=\left\langle A[n^0)\right\rangle$.
\end{proposition}

\begin{proof}
Fix an arbitrary $\delta\in \left\langle A[n^0)\right\rangle$. Then $\delta=\varepsilon^{n^0}_A[i)\cdot \beta^i\alpha^j$ for some $i,j\in\omega$. Lemma~\ref{lemma-2.12} implies that $\beta^k\cdot\varepsilon^{n^0}_A[i)\cdot\alpha^k=\varepsilon^{n^0}_A[i+k)$ for any $k\in\omega$. This implies that for any $k\in\omega$ we have that
\begin{align*}
  \beta^k\cdot\delta\cdot\alpha^k&= \beta^k\cdot\varepsilon^{n^0}_A[i)\cdot \beta^i\alpha^j\cdot\alpha^k=\\
  &= \beta^k\cdot\varepsilon^{n^0}_A[i)\cdot \alpha^k\beta^k\cdot \beta^i\alpha^j\cdot\alpha^k=\\
  &= \varepsilon^{n^0}_A[i+k)\cdot \beta^{i+k}\alpha^{j+k}\in \left\langle A[n^0)\right\rangle,
\end{align*}
and hence ${\downarrow_{\ll}}\delta\in \left\langle A[n^0)\right\rangle$. This implies the equality ${\downarrow_{\ll}}\left\langle A[n^0)\right\rangle=\left\langle A[n^0)\right\rangle$.
\end{proof}

\begin{lemma}\label{lemma-2.14}
Let $\varepsilon^{n^0_1}_{A_1}[i_1)$ and $\varepsilon^{n^0_2}_{A_2}[i_2)$ be idempotents of the monoid $\mathbf{I}\mathbb{N}_{\infty}$. Then:
\begin{enumerate}
  \item\label{lemma-2.14-1} $\varepsilon^{n^0_1}_{A_1}[i_1)\beta^{i_1}\alpha^{i_1}$ is the canonical representation of $\varepsilon^{n^0_1}_{A_1}[i_1)$;
  \item\label{lemma-2.14-2} $\varepsilon^{n^0_1}_{A_1}[i_1)\ll\varepsilon^{n^0_2}_{A_2}[i_2)$ if and only if $A_1=A_2$, $n^0_1=n^0_2$ and $i_1\geqslant i_2$;
  \item\label{lemma-2.14-3} if $\delta_1=\varepsilon^{n^0_1}_{A_1}[i_1)\beta^{i_1}\alpha^{j_1}$ and $\delta_2=\varepsilon^{n^0_2}_{A_1}[i_2)\beta^{i_2}\alpha^{j_2}$ are canonical representations, then  $\delta_1\ll\delta_2$ if and only if $A_1=A_2$, $n^0_1=n^0_2$ and $i_1-i_2=j_1-j_2=k$ for some $k\in\omega$.
\end{enumerate}
\end{lemma}

\begin{proof}
Statement \eqref{lemma-2.14-1} is trivial because $i_1=\min \mathrm{dom}\varepsilon^{n^0_1}_{A_1}[i_1)=\min \mathrm{ran}\varepsilon^{n^0_1}_{A_1}[i_1)$.

\eqref{lemma-2.14-2}
By \eqref{lemma-2.14-1}, $\varepsilon^{n^0_1}_{A_1}[i_1)\beta^{i_1}\alpha^{i_1}$ and $\varepsilon^{n^0_2}_{A_2}[i_2)\beta^{i_2}\alpha^{i_2}$ are the canonical representations of idempotents  $\varepsilon^{n^0_1}_{A_1}[i_1)$ and $\varepsilon^{n^0_2}_{A_2}[i_2)$, respectively.

The relation $\varepsilon^{n^0_1}_{A_1}[i_1)\ll\varepsilon^{n^0_2}_{A_2}[i_2)$ holds if and only if $\varepsilon^{n^0_1}_{A_1}[i_1)=\beta^k\cdot\varepsilon^{n^0_2}_{A_2}[i_2)\cdot\alpha^k$ for some $k\in\omega$. By Lemma~\ref{lemma-2.12}\eqref{lemma-2.12-2} we have that
\begin{align*}
  \varepsilon^{n^0_1}_{A_1}[i_1)&=\varepsilon^{n^0_1}_{A_1}[i_1)\beta^{i_1}\alpha^{i_1}=\\
  &=\beta^k\cdot\varepsilon^{n^0_2}_{A_2}[i_2)\beta^{i_2}\alpha^{i_2}\cdot\alpha^k=\\
  &=\varepsilon^{n^0_2}_{A_2}[i_2+k)\beta^{i_2+k}\alpha^{i_2+k},
\end{align*}
and hence by \eqref{lemma-2.14-1}, $\varepsilon^{n^0_1}_{A_1}[i_1)\ll\varepsilon^{n^0_2}_{A_2}[i_2)$ if and only if  $A_1=A_2$, $n^0_1=n^0_2$ and $i_1\geqslant i_2$.

\eqref{lemma-2.14-3}
Implication $(\Leftarrow)$ is obvious.

$(\Rightarrow)$
If $\delta_1\ll\delta_2$ then there exists $k\in\omega$ such that $\delta_1=\beta^k\delta_2\alpha^k$ for some $k\in\omega$. By Lemma~\ref{lemma-2.12}\eqref{lemma-2.12-4} we have that
\begin{align*}
  \delta_1&=\varepsilon^{n^0_1}_{A_1}[i_1)\beta^{i_1}\alpha^{j_1}=\beta^{i_1}\varepsilon^{n^0_1}_{A_1}[0)\alpha^{j_1}; \\
  \delta_2&=\varepsilon^{n^0_2}_{A_2}[i_2)\beta^{i_2}\alpha^{j_2}=\beta^{i_2}\varepsilon^{n^0_2}_{A_2}[0)\alpha^{j_2},
\end{align*}
and since every element $\gamma$ of the monoid $\mathbf{I}\mathbb{N}_{\infty}$ has the unique canonical representation $\gamma=\varepsilon^{n^0}_{A}[i)\beta^i\alpha^j$, the equalities
\begin{align*}
  \beta^{i_1}\varepsilon^{n^0_1}_{A_1}[0)\alpha^{j_1}&=\delta_1=\\
    &=\beta^k\delta_2\alpha^k=\\
    &=\beta^k\beta^{i_2}\varepsilon^{n^0_2}_{A_2}[0)\alpha^{j_2}\alpha^k=\\
    &=\beta^{i_2+k}\varepsilon^{n^0_2}_{A_2}[0)\alpha^{j_2+k},
\end{align*}
imply that $A_1=A_2$, $n^0_1=n^0_2$ and $i_1-i_2=j_1-j_2=k$.
\end{proof}

\begin{corollary}\label{corollary-2.15}
If $\gamma$ and $\delta$ are comparable elements of the monoid $\mathbf{I}\mathbb{N}_{\infty}$ with respect to the partial order $\ll$ then either there exist a non-empty finite subset $A\subset \mathbb{N}$ such that $\min A=1$ and a positive integer $n^0\geqslant 2+\max A$ such that $\gamma,\delta\in \left\langle A[n^0)\right\rangle$ or $\gamma,\delta\in \mathscr{C}_\mathbb{N}$.
\end{corollary}

\begin{lemma}\label{lemma-2.16}
Let $S$ be a subsemigroup of $\mathbf{I}\mathbb{N}_{\infty}$ which contains $\mathscr{C}_\mathbb{N}$.
Let $\gamma_0\in S$ and  $\beta^{i_1}\alpha^{j_1},\beta^{i_2}\alpha^{j_2}\in S$ be such that $\gamma_0, \beta^{i_1}\alpha^{j_1}\cdot\gamma_0\cdot\beta^{i_2}\alpha^{j_2}\in\left\langle A[n^0)\right\rangle$, where $A$ is a non-empty finite subset of $\mathbb{N}$ such that $\min A=1$ and a positive integer $n^0\geqslant 2+\max A$ or $\left\langle A[n^0)\right\rangle=\mathscr{C}_\mathbb{N}$. Then
\begin{equation*}
\beta^{i_1}\alpha^{j_1}\cdot\eta\cdot\beta^{i_2}\alpha^{j_2}\in {\downarrow_{\ll}}(\beta^{i_1}\alpha^{j_1}\cdot\gamma_0\cdot\beta^{i_2}\alpha^{j_2}) \cap\left\langle A[n^0)\right\rangle
\end{equation*}
for all $\eta\in{\downarrow_{\ll}}\gamma_0$.
\end{lemma}

\begin{proof}
By Proposition~\ref{proposition-2.11} if $\left\langle A[n^0)\right\rangle\cap\mathscr{C}_\mathbb{N}\neq\varnothing$ then $\left\langle A[n^0)\right\rangle=\mathscr{C}_\mathbb{N}$.

In the case when $\left\langle A[n^0)\right\rangle=\mathscr{C}_\mathbb{N}$ by Proposition~\ref{proposition-2.7}  the restriction of the partial order $\ll$ onto $\mathscr{C}_\mathbb{N}$ is the natural partial order on $\mathscr{C}_\mathbb{N}$, and next we apply Proposition~1.4.7 of \cite{Lawson-1998}.

Suppose that $\left\langle A[n^0)\right\rangle\subseteq S\setminus\mathscr{C}_\mathbb{N}$ and let $\gamma_0=\varepsilon^{n^0}_A[i_0)\beta^{i_0}\alpha^{j_0}$ be the canonical representation of $\gamma_0$.
Since $\gamma_0, \beta^{i_1}\alpha^{j_1}\cdot\gamma_0\cdot\beta^{i_2}\alpha^{j_2}\in\left\langle A[n^0)\right\rangle$ we have that $j_1\leqslant i_0$ and $i_2\leqslant j_0$.
Lemma~\ref{lemma-2.12} implies that
\begin{align*}
  \beta^{i_1}\alpha^{j_1}\cdot\gamma_0\cdot\beta^{i_2}\alpha^{j_2}
   &=\beta^{i_1}\alpha^{j_1}\cdot\varepsilon^{n^0}_A[i_0)\beta^{i_0}\alpha^{j_0}\cdot\beta^{i_2}\alpha^{j_2} \\
   &=\beta^{i_1}\cdot\varepsilon^{n^0}_A[i_0-j_1)\alpha^{j_1}\beta^{i_0}\alpha^{j_0}\beta^{i_2}\alpha^{j_2}= \\
   &=\varepsilon^{n^0}_A[i_0-j_1+i_1)\beta^{i_1}\alpha^{j_1}\beta^{i_0}\alpha^{j_0}\beta^{i_2}\alpha^{j_2}= \\
   &=\varepsilon^{n^0}_A[i_0-j_1+i_1)\beta^{i_0-j_1+i_1}\alpha^{j_0-i_2+j_2}.
\end{align*}
Fix an arbitrary $k\in\omega$ and put $\eta=b^{k}\gamma_0\alpha^{k}$. Then by Lemma~\ref{lemma-2.12} for any $k\geqslant k_0$ we have that
\begin{align*}
  \beta^{k}\gamma_0\alpha^{k}
    &=\beta^{k}\varepsilon^{n^0}_A[i_0)\beta^{i_0}\alpha^{j_0}\alpha^{k}=\\
    &=\varepsilon^{n^0}_A[i_0+k)\beta^{k}\beta^{i_0}\alpha^{j_0}\alpha^{k}=\\
    &=\varepsilon^{n^0}_A[i_0+k)\beta^{i_0+k}\alpha^{j_0+k}
\end{align*}
and since $j_1\leqslant i_0$ and $i_2\leqslant j_0$ we get that
\begin{align*}
  \beta^{i_1}\alpha^{j_1}\cdot \beta^{k}\gamma_0\alpha^{k}\cdot\beta^{i_2}\alpha^{j_2}
   &= \beta^{i_1}\alpha^{j_1}\cdot\varepsilon^{n^0}_A[i_0+k)\beta^{i_0+k}\alpha^{j_0+k}\cdot\beta^{i_2}\alpha^{j_2}=\\
   &=\beta^{i_1}\cdot\varepsilon^{n^0}_A[i_0-j_1+k)\alpha^{j_1}\beta^{i_0+k}\alpha^{j_0+k}\beta^{i_2}\alpha^{j_2}=\\
   &=\varepsilon^{n^0}_A[i_0-j_1+i_1+k)\beta^{i_1}\alpha^{j_1}\beta^{i_0+k}\alpha^{j_0+k}\beta^{i_2}\alpha^{j_2}=\\
   &=\varepsilon^{n^0}_A[i_0-j_1+i_1+k)\beta^{i_0-j_1+i_1+k}\alpha^{j_0-i_2+j_2+k}.
\end{align*}
Lemma~\ref{lemma-2.14}\eqref{lemma-2.14-3} implies that $\beta^{i_1}\alpha^{j_1}\cdot \eta\cdot\beta^{i_2}\alpha^{j_2}\ll \beta^{i_1}\alpha^{j_1}\cdot \gamma_0\cdot\beta^{i_2}\alpha^{j_2}$, and hence
\begin{equation*}
\beta^{i_1}\alpha^{j_1}\cdot\eta\cdot\beta^{i_2}\alpha^{j_2}\in {\downarrow_{\ll}}(\beta^{i_1}\alpha^{j_1}\cdot\gamma_0\cdot\beta^{i_2}\alpha^{j_2}) \cap\left\langle A[n^0)\right\rangle
\end{equation*}
for all $\eta\in{\downarrow_{\ll}}\gamma_0$.
\end{proof}

\section{On locally compact submonoids of $\mathbf{I}\mathbb{N}_{\infty}$ with adjoined zero}\label{sec-3}

In this section we assume that $S$ is a submonoid of the semigroup $\mathbf{I}\mathbb{N}_{\infty}$  which contains $\mathscr{C}_\mathbb{N}$. By $S^{\boldsymbol{0}}$ we denote $S$ with the adjoined zero $\boldsymbol{0}$.

\begin{definition}[\!\cite{Chuchman-Gutik-2010}]\label{definitio-3.1}
We shall say that a semigroup $S$ has:
\begin{itemize}
    \item an $\textsf{F}$-\emph{property} if for every $a,b,c,d\in S^1$ the sets $\{x\in S\mid a\cdot x=b\}$ and $\{x\in S\mid x\cdot
     c=d\}$ are finite;

    \item an $\textsf{FS}$-\emph{property} if $S$ is simple and has $\textsf{F}$-property.
\end{itemize}
\end{definition}

\begin{proposition}\label{proposition-3.2}
If $S^{\boldsymbol{0}}$ is a Hausdorff Baire semitopological semigroup, then $S$ is a discrete subspace of $S^{\boldsymbol{0}}$.
\end{proposition}

\begin{proof}
Since $S^{\boldsymbol{0}}$ is a Hausdorff space, $S$ is an open subspace of $S^{\boldsymbol{0}}$ and hence the space $S$ is Baire.

By Proposition 2.2 of \cite{Gutik-Repovs-2011} the monoid $\mathscr{I}_{\infty}^{\!\nearrow}(\mathbb{N})$ has $\textsf{F}$-property. This implies that $\mathbf{I}\mathbb{N}_{\infty}$ has  $\textsf{F}$-property, and hence $S$ has  $\textsf{F}$-property, too. By Theorem~5 of \cite{Gutik-Savchuk-2019} the monoid $S$ is simple. Then by Theorem~5 of \cite{Chuchman-Gutik-2010} every shift-continuous topology on $S$ is discrete.
\end{proof}

\begin{corollary}\label{corollary-3.3}
If $S^{\boldsymbol{0}}$ is a Hausdorff locally compact semitopological semigroup, then $S$ is a discrete subspace of $S^{\boldsymbol{0}}$.
\end{corollary}

Corollary~\ref{corollary-3.3} implies that every open neighbourhood  $U(\boldsymbol{0})$ of  $\boldsymbol{0}$ in Hausdorff locally compact semitopological semigroup $S^{\boldsymbol{0}}$ is a closed subset, i.e., the closure of an open neighbourhood  $U(\boldsymbol{0})$ of $\boldsymbol{0}$ coincides with $U(\boldsymbol{0})$.
These arguments and Corollary~\ref{corollary-3.3} imply the following lemma.

\begin{lemma}\label{lemma-3.4}
Let $\tau$ be a non-discrete Hausdorff locally compact shift-continuous topology on the semigroup $S^{\boldsymbol{0}}$. Then for any compact-and-open neighbourhoods $U(\boldsymbol{0})$ and $V(\boldsymbol{0})$ of $\boldsymbol{0}$ in $(S^{\boldsymbol{0}},\tau)$ both sets $U(\boldsymbol{0})\setminus V(\boldsymbol{0})$ and $V(\boldsymbol{0})\setminus U(\boldsymbol{0})$ are finite.
\end{lemma}

Later in all statements by $U(\boldsymbol{0})$ we denote any compact-and-open neighbourhood of zero  in $(S^{\boldsymbol{0}},\tau)$.

\begin{lemma}\label{lemma-3.5}
Let $\tau$ be a non-discrete Hausdorff locally compact shift-continuous topology on the semigroup $S^{\boldsymbol{0}}$. Then there exists $\gamma^*\in S$ such that ${\downarrow_{\ll}}\gamma_0\cap U(\boldsymbol{0})$ is infinite.
\end{lemma}

\begin{proof}
Suppose to the contrary that the set ${\downarrow_{\ll}}\gamma\cap U(\boldsymbol{0})$ is finite for any $\gamma\in S$ and any neighbourhood $U(\boldsymbol{0})$ of $\boldsymbol{0}$. This implies that the set ${\downarrow_{\ll}}\gamma\cap U(\boldsymbol{0})$ is non-empty for some infinitely many $\gamma\in S$. Hence for such elements $\gamma$ the set ${\downarrow_{\ll}}\gamma\cap U(\boldsymbol{0})$ contains the minimum and the maximum elements with the respect to the order $\ll$ on $S$. By separate continuity of the semigroup operation in $S^{\boldsymbol{0}}$ there exists a compact-and-open neighbourhood $V(\boldsymbol{0})\subseteq U(\boldsymbol{0})$ of  $\boldsymbol{0}$ in $(S^{\boldsymbol{0}},\tau)$ such that $\beta\cdot V(\boldsymbol{0})\cdot\alpha\subseteq U(\boldsymbol{0})$. Then the above arguments and Lemma~\ref{lemma-2.9}\eqref{lemma-2.9-1} imply that the set $U(\boldsymbol{0})\setminus V(\boldsymbol{0})$ is infinite, which contradicts Lemma~\ref{lemma-3.4}. Hence there exists $\gamma_0\in S$ such that the set ${\downarrow_{\ll}}\gamma_0\cap U(\boldsymbol{0})$ is infinite.
\end{proof}

\begin{proposition}\label{proposition-3.5}
Let $\tau$ be a non-discrete Hausdorff locally compact shift-continuous topology on the semigroup $S^{\boldsymbol{0}}$. Then there exists $\gamma^*\in S$ such that ${\downarrow_{\ll}}\gamma^*\subset U(\boldsymbol{0})$.
\end{proposition}

\begin{proof}
By Lemma~\ref{lemma-3.5} there exists  $\gamma_0\in S$ such that the set ${\downarrow_{\ll}}\gamma_0\cap U(\boldsymbol{0})$ is infinite. We claim that the set ${\downarrow_{\ll}}\gamma_0\setminus U(\boldsymbol{0})$ is finite. Suppose to the contrary that there exists a compact-and-open neighbourhood $V(\boldsymbol{0})$ of $\boldsymbol{0}$ in $(S^{\boldsymbol{0}},\tau)$ such that the set ${\downarrow_{\ll}}\gamma_0\setminus  V(\boldsymbol{0})$ is infinite. By the separate continuity of the semigroup operation in $S^{\boldsymbol{0}}$ there exists a compact-and-open neighbourhood $W(\boldsymbol{0})\subseteq V(\boldsymbol{0})$ of  $\boldsymbol{0}$ in $(S^{\boldsymbol{0}},\tau)$ such that $\beta\cdot W(\boldsymbol{0})\cdot\alpha\subseteq V(\boldsymbol{0})$. Then infiniteness of ${\downarrow_{\ll}}\gamma_0\setminus V(\boldsymbol{0})$ and Lemma~\ref{lemma-2.9}\eqref{lemma-2.9-1} imply that the set $V(\boldsymbol{0})\setminus W(\boldsymbol{0})$ is infinite, which contradicts Lemma~\ref{lemma-3.4}. Hence the set ${\downarrow_{\ll}}\gamma_0\setminus V(\boldsymbol{0})$ is finite, and by Lemma~\ref{lemma-3.4} the set ${\downarrow_{\ll}}\gamma_0\setminus U(\boldsymbol{0})$ is finite, as well. By Lemma~\ref{lemma-2.9} there exists $\gamma^*\in {\downarrow_{\ll}}\gamma_0$ such that ${\downarrow_{\ll}}\gamma^*\subset  U(\boldsymbol{0})$.
\end{proof}

\begin{proposition}\label{proposition-3.6}
Let $\tau$ be a non-discrete Hausdorff locally compact shift-continuous topology on the semigroup $S^{\boldsymbol{0}}$. Then there exists  a subset of the form $\left\langle A[n^0)\right\rangle$ in $S$ such that the set  $\left\langle A[n^0)\right\rangle\setminus U(\boldsymbol{0})$ is finite.
\end{proposition}

\begin{proof}
By Proposition~\ref{proposition-3.5}  there exists $\gamma_0\in S$ such that ${\downarrow_{\ll}}\gamma_0\subset U(\boldsymbol{0})$. By Propositions~\ref{proposition-2.4} and~\ref{proposition-2.11} there exists  a subset of the form $\left\langle A[n^0)\right\rangle$ in $S$ such that $\gamma_0\in \left\langle A[n^0)\right\rangle$, and moreover Proposition~\ref{proposition-2.13} implies that ${\downarrow_{\ll}}\gamma_0\subset \left\langle A[n^0)\right\rangle$.

Suppose that $\gamma_0\in \mathscr{C}_\mathbb{N}$. Then we have that $\left\langle A[n^0)\right\rangle=\mathscr{C}_\mathbb{N}$. By Corollary~\ref{corollary-3.3} every point of $S$ is isolated in $S^{\boldsymbol{0}}$, and hence $\mathscr{C}_\mathbb{N}^{\boldsymbol{0}}=\mathscr{C}_\mathbb{N}\sqcup\{\boldsymbol{0}\}$ is a closed subspace of $S^{\boldsymbol{0}}$. Theorem 3.3.8 of \cite{Engelking-1989} implies that the space $\mathscr{C}_\mathbb{N}^{\boldsymbol{0}}$ with the induced topology from $S^{\boldsymbol{0}}$ is locally compact. Proposition~\ref{proposition-3.5} and Theorem~1 of \cite{Gutik-2015} imply that the space $S^{\boldsymbol{0}}$ is compact. This implies that the set  $\left\langle A[n^0)\right\rangle\setminus U(\boldsymbol{0})$ is finite, because $\left\langle A[n^0)\right\rangle=\mathscr{C}_\mathbb{N}$ in the case when an element $\gamma_0$ of the set $\left\langle A[n^0)\right\rangle$ belongs to $\mathscr{C}_\mathbb{N}$.

Suppose that $\gamma_0\in S\setminus\mathscr{C}_\mathbb{N}$ and let $\gamma_0=\varepsilon^{n^0}_A[i_0)\beta^{i_0}\alpha^{j_0}$ be the canonical representation of $\gamma_0$. Fix an arbitrary $\gamma\in \left\langle A[n^0)\right\rangle$ with the canonical representation $\delta=\varepsilon^{n^0}_A[i)\beta^{i}\alpha^{j}$. By Lemma~\ref{lemma-2.12}\eqref{lemma-2.12-4} we have that
\begin{align*}
  \beta^i\alpha^{i_0}\cdot \gamma_0\cdot \beta^{j_0}\alpha^j
   &=\beta^i\alpha^{i_0}\cdot \varepsilon^{n^0}_A[i_0)\beta^{i_0}\alpha^{j_0}\cdot \beta^{j_0}\alpha^j= \\
   &=\beta^i\varepsilon^{n^0}_A[i_0-i_0)\alpha^{i_0}\beta^{i_0}\alpha^{j_0}\beta^{j_0}\alpha^j= \\
   &=\beta^i\varepsilon^{n^0}_A[0)(\alpha^{i_0}\beta^{i_0})(\alpha^{j_0}\beta^{j_0})\alpha^j= \\
   &=\beta^i\varepsilon^{n^0}_A[0)\alpha^j= \\
   &=\varepsilon^{n^0}_A[i)\beta^i\alpha^j= \\
   &=\delta.
\end{align*}
By the separate continuity of the semigroup operation in $(S^{\boldsymbol{0}},\tau)$ there exists a compact-and-open neighbourhood $V_{\gamma}(\boldsymbol{0})$ of $\boldsymbol{0}$ such that $\beta^i\alpha^{i_0}\cdot V_{\gamma}(\boldsymbol{0})\cdot \beta^{j_0}\alpha^j\subseteq U(\boldsymbol{0})$.  Lemma~\ref{lemma-3.4} implies that the set $U(\boldsymbol{0})\setminus V(\boldsymbol{0})$ is finite. Since ${\downarrow_{\ll}}\gamma_0\subset U(\boldsymbol{0})$ there exists $\gamma_1\ll\gamma_0$ such that ${\downarrow_{\ll}}\gamma_1\subset V(\boldsymbol{0})$. By Lemma~\ref{lemma-2.14}\eqref{lemma-2.14-3} there exists $k\in\omega$ such that $\gamma_1=\varepsilon^{n^0}_A[i_0+k)\beta^{i_0+k}\alpha^{j_0+k}$ is the canonical representation of $\gamma_1$. Then we get that
\begin{align*}
  \beta^i\alpha^{i_0}\cdot \gamma_1\cdot \beta^{j_0}\alpha^j
   &=\beta^i\alpha^{i_0}\cdot \varepsilon^{n^0}_A[i_0+k)\beta^{i_0+k}\alpha^{j_0+k}\cdot \beta^{j_0}\alpha^j= \\
   &=\beta^i\varepsilon^{n^0}_A[i_0-i_0+k)\alpha^{i_0}\beta^{i_0+k}\alpha^{j_0+k}\beta^{j_0}\alpha^j= \\
   &=\beta^i\varepsilon^{n^0}_A[k)(\alpha^{i_0}\beta^{i_0})\beta^k\alpha^k(\alpha^{j_0}\beta^{j_0})\alpha^j= \\
   &=\beta^i\varepsilon^{n^0}_A[k)\beta^k\alpha^k\alpha^j= \\
   &=\varepsilon^{n^0}_A[i+k)\beta^{i+k}\alpha^{j+k}\ll\\
   &\ll\delta,
\end{align*}
and hence $\beta^i\alpha^{i_0}\cdot {\downarrow_{\ll}}\gamma_1\cdot \beta^{j_0}\alpha^j= {\downarrow_{\ll}}(\beta^i\alpha^{i_0}\cdot \gamma_1\cdot \beta^{j_0}\alpha^j)$. This implies that for any neighbourhood $U(\boldsymbol{0})$  of the zero $\boldsymbol{0}$ in $(S^{\boldsymbol{0}},\tau)$ and any $\gamma\in \left\langle A[n^0)\right\rangle$ there exists $\delta \in\left\langle A[n^0)\right\rangle$  such that $\delta\ll\gamma$ and ${\downarrow_{\ll}}\delta\subseteq U(\boldsymbol{0})$.

Suppose to the contrary that the set  $\left\langle A[n^0)\right\rangle\setminus U(\boldsymbol{0})$ is infinite. Then the above arguments imply that there exists an infinite sequence $\left\{\gamma_p\right\}_{p\in\omega}\subseteq \left\langle A[n^0)\right\rangle\setminus U(\boldsymbol{0})$ such that all elements of the sequence $\left\{\gamma_p\right\}_{p\in\omega}$ are incomparable with respect to the partial order $\ll$ on $S^{\boldsymbol{0}}$ and ${\downarrow_{\ll}}\left\{\gamma_p\right\}_{p\in\omega}\setminus U(\boldsymbol{0})=\left\{\gamma_p\right\}_{p\in\omega}$. By the separate continuity of the semigroup operation in $(S^{\boldsymbol{0}},\tau)$ there exists a compact-and-open neighbourhood $W(\boldsymbol{0})$ of $\boldsymbol{0}$ such that $\alpha\cdot W(\boldsymbol{0})\cdot \beta\subseteq U(\boldsymbol{0})$. The above arguments imply that  $\left\{\gamma_p\right\}_{p\in\omega}\cap U(\boldsymbol{0})=\varnothing$ and $\left\{\beta\gamma_p\alpha\right\}_{p\in\omega} \subset U(\boldsymbol{0})$. But $\left\{\beta\gamma_p\alpha\right\}_{p\in\omega} \cap W(\boldsymbol{0})=\varnothing$ and hence the set $U(\boldsymbol{0})\setminus W(\boldsymbol{0})$ is infinite, which contradicts Lemma~\ref{lemma-3.4}. The obtained contradiction completes the proof of the proposition.
\end{proof}

\begin{proposition}\label{proposition-3.7}
Let $\tau$ be a non-discrete Hausdorff locally compact shift-continuous topology on the semigroup $S^{\boldsymbol{0}}$. If there exists  a subset of the form $\left\langle A[n^0)\right\rangle$ in $S$ such that the set  $\left\langle A[n^0)\right\rangle\cap U(\boldsymbol{0})$ is infinite then the set $\left\langle A[n^0)\right\rangle\setminus U(\boldsymbol{0})$ is finite.
\end{proposition}

\begin{proof}
We claim that there exists $\gamma_0\in \left\langle A[n^0)\right\rangle$ such that ${\downarrow_{\ll}}\gamma_0\subset U(\boldsymbol{0})$. Suppose the contrary: the set ${\downarrow_{\ll}}\gamma_0\setminus U(\boldsymbol{0})$ is infinite for any $\gamma_0\in S$. By the separate continuity of the semigroup operation in $S^{\boldsymbol{0}}$ there exists a neighbourhood $W(\boldsymbol{0})\subseteq V(\boldsymbol{0})$ of  $\boldsymbol{0}$ in  $(S^{\boldsymbol{0}},\tau)$ such that $\beta\cdot W(\boldsymbol{0})\cdot\alpha\subseteq V(\boldsymbol{0})$. Then  the set $U(\boldsymbol{0})\setminus V(\boldsymbol{0})$ is infinite, which contradicts Lemma~\ref{lemma-3.4}. Hence ${\downarrow_{\ll}}\gamma_0\subset U(\boldsymbol{0})$ for some $\gamma_0\in \left\langle A[n^0)\right\rangle$. Next, in a similar way as in the proof of Proposition~\ref{proposition-3.6} it can be show that the set $\left\langle A[n^0)\right\rangle\setminus U(\boldsymbol{0})$ is finite.
\end{proof}

\begin{proposition}\label{proposition-3.8}
Let $\tau$ be a non-discrete Hausdorff locally compact shift-continuous topology on the semigroup $S^{\boldsymbol{0}}$. Then for any subset of the form $\left\langle A[n^0)\right\rangle$ in $S$ the set $\left\langle A[n^0)\right\rangle\setminus U(\boldsymbol{0})$ is finite.
\end{proposition}

\begin{proof}
Fix an arbitrary neighbourhood $U(\boldsymbol{0})$ of the zero $\boldsymbol{0}$ in $(S^{\boldsymbol{0}},\tau)$. First we shall show that the set $\left\langle A[n^0)\right\rangle\setminus U(\boldsymbol{0})$ is finite for $\left\langle A[n^0)\right\rangle=\mathscr{C}_\mathbb{N}$. By Proposition~\ref{proposition-3.6} there exists  a subset of the form $\left\langle A[n^0)\right\rangle$ in $S$ such that the set  $\left\langle A[n^0)\right\rangle\setminus U(\boldsymbol{0})$ is finite. Then there exists a positive integer $k_U$ such that the element $\gamma_l=\varepsilon^{n^0}_A[0)\beta^{0}\alpha^{l}\in \left\langle A[n^0)\right\rangle$ belongs to $U(\boldsymbol{0})$ for any $l\geqslant k_U$. Since the set $U(\boldsymbol{0})\setminus W(\boldsymbol{0})$ finite for any compact-and-open neighbourhood $W(\boldsymbol{0})$ of $\boldsymbol{0}$, there exists a positive integer $k_W\geqslant k_U$ such that the element $\gamma_l=\varepsilon^{n^0}_A[0)\beta^{0}\alpha^{l}\in \left\langle A[n^0)\right\rangle$ belongs to $W(\boldsymbol{0})$ for any $l\geqslant k_W$.

By the separate continuity of the semigroup operation in $S^{\boldsymbol{0}}$ there exists a compact-and-open neighbourhood $W(\boldsymbol{0})\subseteq U(\boldsymbol{0})$ of the zero $\boldsymbol{0}$ in $(S^{\boldsymbol{0}},\tau)$ such that $\alpha^{n^0}\cdot W(\boldsymbol{0})\subseteq U(\boldsymbol{0})$. Since $\mathrm{ran}(\alpha^{n^0})\subseteq \mathrm{dom}(\varepsilon^{n^0}_A[0))$ and $\varepsilon^{n^0}_A[0)$ is an idempotent of $S$, we get that $\alpha^{n^0}\cdot\varepsilon^{n^0}_A[0)=\alpha^{n^0}$, and hence
\begin{equation*}
  \alpha^{n^0}\cdot \gamma_l=\alpha^{n^0}\cdot\varepsilon^{n^0}_A[0)\beta^{0}\alpha^{l}= \alpha^{n^0}\beta^{0}\alpha^{l}= \alpha^{n^0+l}\in U(\boldsymbol{0}),
\end{equation*}
which implies that the set $\mathscr{C}_\mathbb{N}\cap U(\boldsymbol{0})$ is infinite. Then by Proposition~\ref{proposition-3.7} the set $\mathscr{C}_\mathbb{N}\setminus U(\boldsymbol{0})$ is finite.

Fix any subset of the form $\left\langle A[n^0)\right\rangle$ in $S$ distinct from $\mathscr{C}_\mathbb{N}$. By the separate continuity of the semigroup operation in $S^{\boldsymbol{0}}$ there exists a neighbourhood $V(\boldsymbol{0})\subseteq U(\boldsymbol{0})$ of the zero $\boldsymbol{0}$ in $(S^{\boldsymbol{0}},\tau)$ such that $\varepsilon^{n^0}_A[0)\cdot V(\boldsymbol{0})\subseteq U(\boldsymbol{0})$. By the previous part of the proof the set $\mathscr{C}_\mathbb{N}\setminus V(\boldsymbol{0})$ is finite, and hence there exists a positive integer $k_V$ such that the element $\gamma_l=\beta^{l}\in \mathscr{C}_\mathbb{N}$ belongs to $V(\boldsymbol{0})$ for any $l\geqslant k_V$. This implies that the neighbourhood $U(\boldsymbol{0})$ contains the infinite set $\left\{\varepsilon^{n^0}_A[0)\cdot \beta^{l}\colon l\geqslant k_V\right\}$. By Proposition~\ref{proposition-3.7} the set $\left\langle A[n^0)\right\rangle\setminus U(\boldsymbol{0})$ is finite.
\end{proof}

Lemma~\ref{lemma-3.9} shows that on a semigroup $S^{\boldsymbol{0}}$, where $S$ be a some subsemigroup of $\mathbf{I}\mathbb{N}_{\infty}$ which contains $\mathscr{C}_\mathbb{N}$, there exists a Hausdorff topology $\tau_{\operatorname{\textsf{Ac}}}$ such that $(S^{\boldsymbol{0}},\tau_{\operatorname{\textsf{Ac}}})$ is a compact semitopological semigroup.

\begin{lemma}\label{lemma-3.9}
Let $T$ be a semigroup with $\textsf{F}$-property and $T^0$ be the semigroup $T$ with adjoined zero. Let $\tau_{\operatorname{\textsf{Ac}}}$ be the topology on $T^0$  such that
\begin{itemize}
  \item[$(i)$] every element of $T$ is an isolated point in the space $(T^{0},\tau_{\operatorname{\textsf{Ac}}})$;
  \item[$(ii)$] the family $\mathscr{B}({0})=\left\{U\subseteq T^{{0}}\colon U\ni {0} \hbox{~and~} T^{{0}}\setminus U \hbox{~is finite}\right\}$ determines a base of the topology $\tau_{\operatorname{\textsf{Ac}}}$ at zero ${0}\in T^{{0}}$.
\end{itemize}
Then $(T^{0},\tau_{\operatorname{\textsf{Ac}}})$ is a Hausdorff compact semitopological  semigroup.
\end{lemma}

\begin{proof}
Since all points of the semigroup $T$ are isolated in $(T^0,\tau_{\operatorname{\textsf{Ac}}})$ it is sufficient to show that the semigroup operation in $(T^0,\tau_{\operatorname{\textsf{Ac}}})$ is separately continuous in the following two cases:
\begin{equation*}
  \gamma\cdot{0}={0} \qquad \hbox{and} \qquad {0}\cdot\gamma={0}, \qquad \hbox{for~} \gamma\in T.
\end{equation*}
Fix an arbitrary open set $U({0})\in\mathscr{B}({0})$. Then $K=T^{{0}}\setminus U({0})$ is finite. Let
\begin{equation*}
  K_\gamma=\left\{\delta\in T\colon \gamma\cdot\delta\in K\right\}\cup\left\{\delta\in T\colon\delta\cdot\gamma\in K\right\}.
\end{equation*}
Since the semigroup $T$ has $\textsf{F}$-property, the set $K_\gamma$ is finite. Then
\begin{equation*}
  \gamma\cdot V({0})\subseteq U({0}) \qquad \hbox{and} \qquad V({0})\cdot \gamma\subseteq U({0})
\end{equation*}
for $V({0})=U({0})\setminus K_\gamma\in\mathscr{B}({0})$.
\end{proof}

\begin{remark}\label{remark-3.10}
By Corollary~\ref{corollary-3.3} the discrete topology is a unique Hausdorff locally compact shift-continuous topology on a subsemigroup $S$  of $\mathbf{I}\mathbb{N}_{\infty}$  which contains $\mathscr{C}_\mathbb{N}$. So $\tau_{\operatorname{\textsf{Ac}}}$ is the unique compact shift-continuous topology on $S^{\boldsymbol{0}}$.
\end{remark}

\begin{theorem}\label{theorem-3.11}
Let $S$ be a some susemigroup of $\mathbf{I}\mathbb{N}_{\infty}$ which contains $\mathscr{C}_\mathbb{N}$. Then every Hausdorff locally compact shift-continuous topology $\tau$ on the semigroup $S^{\boldsymbol{0}}$ is either compact or discrete.
\end{theorem}

\begin{proof}
If the zero  ${\boldsymbol{0}}$ of the semigroup $S^{\boldsymbol{0}}$ is an isolated point in $(S^{\boldsymbol{0}},\tau)$ then  the space $S$ is locally compact. By Corollary~\ref{corollary-3.3}, $S$ is discrete and hence $(S^{\boldsymbol{0}},\tau)$ is discrete, too.

Suppose that $\boldsymbol{0}$ is a nonisolated point in $(S^{\boldsymbol{0}},\tau)$. Let $U(\boldsymbol{0})$ be any neighbourhood of $\boldsymbol{0}$ in $(S^{\boldsymbol{0}},\tau)$. By Proposition~\ref{proposition-3.8} for any subset of the form $\left\langle A[n^0)\right\rangle$ in $S$ the set $\left\langle A[n^0)\right\rangle\setminus U(\boldsymbol{0})$ is finite. If the semigroup $S$ is the union of a finite family $\left\{ \mathscr{C}_\mathbb{N}, \left\langle A_1[n^0_1)\right\rangle, \ldots,\left\langle A_k[n^0_k)\right\rangle\right\}$ of subsets of $S$, such  that $\min A_i=1$ and $n^0_i\geqslant 2+\max A_i$ for all positive integer $i\leqslant k$, $k\in\mathbb{N}$, then Proposition~\ref{proposition-3.8} implies that $S^{\boldsymbol{0}}\setminus U(\boldsymbol{0})$ is finite, and hence the space $(S^{\boldsymbol{0}},\tau)$ is compact.

Next, we suppose that the semigroup $S$ is the union of an infinite family of distinct sets
\begin{equation*}
\left\{ \mathscr{C}_\mathbb{N}, \left\langle A_{1}[n^0_{1})\right\rangle, \ldots,\left\langle A_{k}[n^0_{i_k})\right\rangle, \ldots\right\}
\end{equation*}
of $S$, such  that $\min A_{i}=1$ and $n^0_{i}\geqslant 2+\max A_{i}$ for ${i}\in \mathbb{N}$.
Suppose to the contrary that there exists a neighbourhood $U(\boldsymbol{0})$ of $\boldsymbol{0}$ in $(S^{\boldsymbol{0}},\tau)$ such that the set $S^{\boldsymbol{0}}\setminus U(\boldsymbol{0})$ is infinite. By Proposition~\ref{proposition-3.8} without loss of generality we may assume that there exists an infinite sequence $\left\{\gamma_i\right\}_{i\in\mathbb{N}}$ of distinct elements of $S^{\boldsymbol{0}}\setminus U(\boldsymbol{0})$ which satisfies the following properties:
\begin{enumerate}
  \item $\gamma_i\in \left\langle A_{i}[n^0_{i})\right\rangle$ for any $i\in\mathbb{N}$;
  \item $\beta\cdot\gamma_i\cdot\alpha\in U(\boldsymbol{0})$ for any $i\in\mathbb{N}$;
  \item if $i\neq j$  then $\left\langle A_{i}[n^0_{i})\right\rangle\cap \left\langle A_{j}[n^0_{j})\right\rangle=\varnothing$.
\end{enumerate}
By the separate continuity of the semigroup operation in $(S^{\boldsymbol{0}},\tau)$ there exists an open neighbourhood $W(\boldsymbol{0})\subset U(\boldsymbol{0})$ of $\boldsymbol{0}$ such that $\alpha\cdot W(\boldsymbol{0})\cdot \beta\subseteq U(\boldsymbol{0})$. Then our assumption implies that $\beta\cdot\gamma_i\cdot\alpha\notin W(\boldsymbol{0})$ for any $i\in\mathbb{N}$. Hence the set $U(\boldsymbol{0})\setminus W(\boldsymbol{0})$ is infinite which contradicts Lemma~\ref{lemma-3.4}. The obtained contradiction implies that the space $(S^{\boldsymbol{0}},\tau)$ is compact.
\end{proof}

Since the bicyclic monoid does not embeds into any Hausdorff compact topological semigroup \cite{Anderson-Hunter-Koch-1965}, Theorem~\ref{theorem-3.11} implies the following corollary.

\begin{corollary}\label{corollary-3.12}
Let $S$ be a some subsemigroup of   $\mathbf{I}\mathbb{N}_{\infty}$ which contains $\mathscr{C}_\mathbb{N}$. Then every Hausdorff locally compact semigroup topology on $S^{\boldsymbol{0}}$ is discrete.
\end{corollary}

\begin{remark}\label{remark-3.13}
In the paper \cite{Gutik-2015} the example that a counterpart of the statement of Corollary~\ref{corollary-3.12} (and hence the statement of Theorem~\ref{theorem-3.11}) does not hold when $\mathscr{C}_\mathbb{N}$ is a \v{C}ech-complete metrizable topological inverse semigroup is constructed.
\end{remark}

Later we need the following trivial lemma, which follows from separate continuity of the semigroup operation in semitopological semigroups.

\begin{lemma}\label{lemma-3.14}
Let $S$ be a Hausdorff semitopological semigroup and $I$ be a compact ideal in $S$. Then the Rees-quotient semigroup $S/I$ with the quotient topology is a Hausdorff semitopological semigroup.
\end{lemma}

\begin{lemma}\label{lemma-3.15}
Let $X$ be a Hausdorff locally compact space and $I$ be a compact subset of $X$. Then there exists an open neighbourhood $U(I)$ of $I$ with the compact closure $\overline{U(I)}$.
\end{lemma}

\begin{proof}
Fix for any $x\in I$ an open neighbourhood $U(x)$ of $x$ in $X$ such that the closure $\overline{U(x)}$ is compact. Then $\left\{U(x)\colon x\in I\right\}$ is an open cover of $I$. Since $I$ is compact, $I\subseteq U(I)=U(x_i)\cup U(x_2)\cup\cdots\cup U(x_k)$ for some finitely many $x_1,x_2,\ldots,x_k\in I$. Then the set $\overline{U(I)}=\overline{U(x_i)}\cup \overline{U(x_2)}\cup\cdots\cup \overline{U(x_k)}$ is compact.
\end{proof}

\begin{theorem}\label{theorem-3.15}
Let $S$ be a some subsemigroup of $\mathbf{I}\mathbb{N}_{\infty}$ which contains $\mathscr{C}_\mathbb{N}$.
Let $(S_I,\tau)$ be a Hausdorff locally compact semitopological semigroup, where $S_I=S\sqcup I$ and $I$ is a compact ideal of $S_I$. Then either $(S_I,\tau)$ is a compact semitopological semigroup or the ideal $I$ open.
\end{theorem}

\begin{proof}
Suppose that $I$ is not open. By Lemma~\ref{lemma-3.14} the Rees-quotient semigroup $S_I/I$ with the quotient topology $\tau_{\operatorname{\textsf{q}}}$ is a semitopological semigroup. Let $\pi\colon S_I\to S_I/I$ be the natural homomorphism which is a quotient map. It is obvious that the Rees-quotient semigroup $S_I/I$ is isomorphic to the semigroup $S^{\boldsymbol{0}}$ and without loss of generality we can assume that $\pi(S_I)=S^{\boldsymbol{0}}$ and the image $\pi(I)$ is zero of $S^{\boldsymbol{0}}$.

By Lemma~\ref{lemma-3.15} there exists an open neighbourhood $U(I)$ of $I$ with the compact closure $\overline{U(I)}$. Since by Corollary~\ref{corollary-3.3} every point of $S$ is isolated in $(S_I,\tau)$ we have that $\overline{U(I)}=U(I)$ and its image $\pi(U(I))$ is compact-and-open neighbourhood of zero in $S^{\boldsymbol{0}}$. Since for any open neighbourhood $V(I)$ of $I$ in $(S_I,\tau)$ the set $\overline{U(I)\cap V(I)}$ is compact, Theorem~\ref{theorem-3.11} implies that $S^{\boldsymbol{0}}\setminus\pi(U(I))$ is finite for any compact-and-open neighbourhood $U(I)$ of $I$ in $(S_I,\tau)$. Then compactness of $I$ implies that $(S_I,\tau)$ is compact as well.
\end{proof}

\begin{corollary}\label{corollary-3.16}
Let $S$ be a some submonoid of the semigroup $\mathbf{I}\mathbb{N}_{\infty}$ which contains $\mathscr{C}_\mathbb{N}$.
If $(S_I,\tau)$ is a locally compact topology topological semigroup, where $S_I=S\sqcup I$ and $I$ is a compact ideal of $S_I$, then the ideal $I$ is open.
\end{corollary}
\medskip
\section*{Acknowledgements}
The authors acknowledge  Alex Ravsky and the referees
for useful important comments and suggestions.


\begin{thebibliography}{XX}
\bibitem{Anderson-Hunter-Koch-1965}
L.~W.~Anderson, R.~P.~Hunter, and R.~J.~Koch,
\emph{Some results on stability in semigroups}.
Trans. Amer. Math. Soc. {\bf 117} (1965), 521--529.



\bibitem{Banakh-Bardyla-Guran-Gutik-Ravsky-2020}
T. Banakh, S. Bardyla, I. Guran, O. Gutik, and A. Ravsky,
\emph{Positive answers for Koch's problem in special cases},
Topol. Algebra Appl. \textbf{8} (2020),  76-87.

\bibitem{Banakh-Dimitrova-Gutik-2009}
T.~Banakh, S.~Dimitrova, and O.~Gutik,
\emph{The Rees-Suschkiewitsch Theorem for simple
topological semigroups}, Mat. Stud. \textbf{31}  (2009), no. 2, 211--218.

\bibitem{Banakh-Dimitrova-Gutik-2010}
T.~Banakh, S.~Dimitrova, and O.~Gutik,
\emph{Embedding the bicyclic semigroup into countably compact topological semigroups},
Topology Appl. \textbf{157} (2010), no.~18, 2803--2814.

\bibitem{Bardyla-2016}
S. Bardyla,
\emph{Classifying locally compact semitopological polycyclic monoids},
Mat. Visn. Nauk. Tov. Im. Shevchenka \textbf{13} (2016), 21--28.

\bibitem{Bardyla-2018}
S. Bardyla,
\emph{On locally compact semitopological graph inverse semigroups},
Mat. Stud. \textbf{49} (2018), no. 1, 19--28.

\bibitem{Bardyla=2021??}
S. Bardyla,
\emph{On topological McAlister semigroups},
J. Pure Appl. Algebra \textbf{227} (2023), no.~4, 107274.

\bibitem{Bardyla-Ravsky-2019}
S. Bardyla and A. Ravsky,
\emph{Closed subsets of compact-like topological spaces},
Appl. Gen. Topol. \textbf{21} (2020), no. 2, 201--214.


\bibitem{Bertman-West-1976}
M.~O.~Bertman and T.~T.~West,
{\it Conditionally compact bicyclic semitopological semigroups},
Proc. Roy. Irish Acad. {\bf A76} (1976), no. 21--23, 219--226.


\bibitem{Bezushchak-2004}
O. Bezushchak,
\emph{On growth of the inverse semigroup of partially defined co-finite automorphisms of integers},
Algebra Discrete Math. (2004), no.~2, 45--55.


\bibitem{Bezushchak-2008}
O. Bezushchak,
\emph{Green's relations of the inverse semigroup of partially defined cofinite isometries of discrete line},
Visn., Ser. Fiz.-Mat. Nauky, Kyiv. Univ. Im. Tarasa Shevchenka (2008), no.~1, 12--16.


\bibitem{Carruth-Hildebrant-Koch-1983}
J.~H.~Carruth, J.~A.~Hildebrant, and  R.~J.~Koch,
\emph{The Theory of Topological Semigroups}, Vol. I, Marcel
Dekker, Inc., New York and Basel, 1983.

\bibitem{Carruth-Hildebrant-Koch-1986}
J.~H.~Carruth, J.~A.~Hildebrant, and  R.~J.~Koch,
\emph{The Theory of Topological Semigroups}, Vol. II, Marcel Dekker,
Inc., New York and Basel, 1986.


\bibitem{Chuchman-Gutik-2010}
I. Ya. Chuchman and O. V. Gutik,
\emph{Topological monoids of almost monotone injective co-finite partial selfmaps of the set of positive integers}.
Carpathian Math. Publ. \textbf{2} (2010), no.~1, 119--132.


\bibitem{Clifford-Preston-1961}
A. H.~Clifford and  G. B.~Preston,
\emph{The Algebraic Theory of Semigroups},
Vol. I., Amer. Math. Soc. Surveys 7, Pro\-vidence, R.I., 1961.

\bibitem{Clifford-Preston-1967}
A. H.~Clifford and  G. B.~Preston,
\emph{The Algebraic Theory of Semigroups},
Vol. II., Amer. Math. Soc. Surveys 7, Providence, R.I., 1967.


\bibitem{Eberhart-Selden-1969}
C. Eberhart and J. Selden,
\emph{ On the closure of the bicyclic semigroup},
Trans. Amer. Math. Soc. {\bf 144} (1969), 115--126.

\bibitem{Engelking-1989}
R.~Engelking,
\emph{General Topology}, 2nd ed.,
Heldermann, Berlin, 1989.


\bibitem{Guran-Kisil-2012}
I. Guran and M. Kisil',
\emph{Pontryagin's alternative for locally compact cancellative monoids},
Visnyk Lviv Univ. Ser. Mech. Math. \textbf{77} (2012), 84--88 (in Ukrainian).

\bibitem{Gutik-2015}
O. Gutik,
\emph{On the dichotomy of a locally compact semitopological bicyclic monoid with adjoined zero},
Visnyk L'viv Univ., Ser. Mech.-Math. \textbf{80} (2015), 33--41.

\bibitem{Gutik-P.Khylynskyi-2021}
O. Gutik and P. Khylynskyi,
\emph{On the monoid of cofinite partial isometries of $\mathbb{N}$ with a bounded finite noise},
Proceedings of the Contemporary Mathematics in Kielce 2020, ed. Szymon Walczak. Jan Kochanowski University in Kielce, Poland. February 24-27, 2021. Sciendo, De Gruyter Poland Sp. z o.o. Warsaw, Poland, 2021, P. 127--144.

\bibitem{Gutik-Maksymyk-2019}
O. V. Gutik and K. M. Maksymyk,
\emph{On a semitopological extended bicyclic semigroup with adjoined zero},
Mat. Metody Fiz.-Mekh. Polya \textbf{62} (2019), no. 4, 28--38.

\bibitem{Gutik-Repovs-2007}
O.~Gutik and D.~Repov\v{s},
\emph{On countably compact $0$-simple topological inverse semigroups},
Semigroup Forum \textbf{75} (2007), no.~2, 464--469.


\bibitem{Gutik-Repovs-2011}
O.~Gutik and D.~Repov\v{s},
\emph{Topological monoids of monotone, injective partial selfmaps of $\mathbb{N}$ having cofinite domain and image},
Stud. Sci. Math. Hungar. \textbf{48} (2011), no.~3, 342--353.

\bibitem{Gutik-Repovs-2012}
O.~Gutik and D.~Repov\v{s},
\emph{On monoids of injective partial selfmaps of integers with cofinite domains and images},
Georgian Math. J. \textbf{19} (2012),  no.~3, 511--532.

\bibitem{Gutik-Repovs-2015}
O.~Gutik and D.~Repov\v{s},
\emph{On monoids of injective partial cofinite selfmaps},
Math. Slovaca \textbf{65} (2015), no.~5,  981--992.


\bibitem{Gutik-Savchuk-2017}
O.~Gutik and A.~Savchuk,
\emph{On the semigroup $\mathbf{ID}_{\infty}$,}
Visn. Lviv. Univ., Ser. Mekh.-Mat. \textbf{83} (2017), 5--19 (in Ukrainian).

\bibitem{Gutik-Savchuk-2018}
O.~Gutik and A.~Savchuk,
\emph{The semigroup of partial co-finite isometries of positive integers,}
 Bukovyn. Mat. Zh. \textbf{6} (2018), no. 1--2,  42--51 (in Ukrainian).

\bibitem{Gutik-Savchuk-2019}
O. Gutik and A. Savchuk,
\emph{On inverse submonoids of the monoid of almost monotone injective co-finite partial selfmaps of positive integers},
Carpathian Math. Publ. \textbf{11} (2019), no. 2, 296--310. 

\bibitem{Gutik-Savchuk-2020}
O. Gutik and A. Savchuk,
\emph{On the monoid of cofinite partial isometries of $\mathbb{N}$ with the usual metric},
Visn. Lviv. Univ., Ser. Mekh.-Mat. \textbf{89} (2020), 17--30.

\bibitem{Haworth-McCoy-1977}
R. C. Haworth and R. A. McCoy,
\emph{Baire spaces},
Dissertationes Math.  \textbf{141}, Warszawa (1977), 73pp.

\bibitem{Hewitt-1956}
E. Hewitt,
\emph{Compact monothetic semigroups},
Duke Math. J. \textbf{23} (1956), no. 3, 447--457.


\bibitem{Hildebrant-Koch-1986}
J.~A.~Hildebrant and R.~J.~Koch,
{\it Swelling actions of $\Gamma$-compact semigroups}, Semigroup
Forum {\bf 33} (1986), 65--85.

\bibitem{Hofmann-1960}
K. H.~Hofmann,
\emph{Topologische Halbgruppen mit dichter submonoger Untenhalbgruppe},
Math. Zeit. \textbf{74} (1960), 232--276.

\bibitem{Hofmann-Mostert-1966}
K. H. Hofmann and P. S. Mostert,
\emph{Elements of compact semigroups},
Co\-lum\-bus: Chas. E. Merrill Co., 1966.


\bibitem{Koch-1957}
R.~J.~Koch,
\emph{On monothetic semigroups},
Proc. Amer. Math. Soc. \textbf{8} (1957), no. 2, 397--401.





\bibitem{Koch-Wallace-1957}
R.~J.~Koch and A.~D.~Wallace,
\emph{Stability in semigroups},
Duke Math. J. \textbf{24} (1957), no. 2, 193--195.


\bibitem{Lawson-1998}
M.~Lawson,
\emph{Inverse Semigroups. The Theory of Partial Symmetries},
Singapore: World Scientific, 1998.



\bibitem{Mokrytskyi-2019}
T. Mokrytskyi,
\emph{On the dichotomy of a locally compact semitopological monoid of order isomorphisms between principal filters of $\mathbb{N}^n$ with adjoined zero},
Visn. Lviv Univ., Ser. Mekh.-Mat. \textbf{87} (2019), 37--45.



\bibitem{Numakura-1952}
K. Numakura,
\emph{On bicompact semigroups},
Math. J. Okayama Univ. \textbf{1} (1952), 99--108.



\bibitem{Petrich-1984}
M.~Petrich,
\emph{Inverse Semigroups},
John Wiley $\&$ Sons, New York, 1984.


\bibitem{Reilly-1966}
N. R. Reilly,
\emph{Bisimple $\omega$-semigroups},
Proc. Glasg. Math. Assoc. \textbf{7} (1966), no. 3, 160--167.

\bibitem{Ruppert-1984}
W.~Ruppert,
\emph{Compact Semitopological Semigroups: An Intrinsic Theory},
Lect. Notes Math., \textbf{1079}, Springer, Berlin, 1984.



\bibitem{Wagner-1952}
V.~V. Wagner,
\textit{Generalized groups},
Dokl. Akad. Nauk SSSR \textbf{84} (1952), 1119--1122 (in Russian).


\bibitem{Weil-1938}
A. Weil.
\textit{L'integration dans les groupes lopologiques et ses applications},
Actualites Scientifiques No. 869, Hermann, Paris, 1938.

\bibitem{Zelenyuk-1988}
E. G. Zelenyuk,
\textit{On Pontryagin's alternative for topological semigroups},
Mat. Zametki \textbf{44} (1988), no. 3, 402--403 (in Russian).

\bibitem{Zelenyuk-2019}
Ye. Zelenyuk,
\emph{A locally compact noncompact monothetic semigroup with identity},
Fund. Math. \textbf{245} (2019), no. 1, 101--107.

\bibitem{Zelenyuk-2020}
Ye. Zelenyuk,
\emph{Larger locally compact monothetic semigroups},
Semigroup Forum \textbf{100} (2020), no. 2, 605--616.

\bibitem{Zelenyuk-Zelenyuk-2020}
Ye. Zelenyuk and Yu. Zelenyuk,
\emph{When a locally compact monothetic semigroup is compact},
J. Group Theory \textbf{23} (2020), no. 6, 983--989.

\end{thebibliography}
\end{document}